\documentclass[11pt]{amsart}

\usepackage{amscd}
\usepackage{amsmath, amssymb}
\usepackage{xcolor}
\usepackage{esint}
\usepackage{amsfonts}
\newcommand{\de}{\partial}
\newcommand{\ii}{\sqrt{-1}}
\newcommand{\db}{\overline{\partial}}

\newcommand{\ddbar}{i \partial \overline{\partial}}
\newcommand{\Ric}{\mathrm{Ric}}
\newcommand{\ov}[1]{\overline{#1}}

\newcommand{\tr}[2]{\textrm{tr}_{#1}{#2}}
\newcommand{\ti}[1]{\tilde{#1}}
\newcommand{\vp}{\varphi}
\newcommand{\vol}{\mathrm{Vol}}

\newcommand{\diam}{\mathrm{diam}}
\newcommand{\hilb}{\mathrm{Hilb}}
\newcommand{\loc}{\mathrm{loc}}
\newcommand{\rcd}{\mathrm{RCD}}
\newcommand{\reg}{\mathrm{reg}}
\newcommand{\sing}{\mathrm{sing}}
\newcommand{\entry}{\mathrm{entry}}
\newcommand{\exit}{\mathrm{exit}}
\newcommand{\supp}{\mathrm{supp}}
\newcommand{\gh}{\mathrm{GH}}
\newcommand{\Sym}{\mathrm{Sym}}
\newcommand{\ve}{\varepsilon}

\renewcommand{\leq}{\leqslant}
\DeclareMathOperator{\can}{can}
\renewcommand{\geq}{\geqslant}
\renewcommand{\le}{\leqslant}
\renewcommand{\ge}{\geqslant}
\renewcommand{\epsilon}{\varepsilon}

\begin{document}
\newtheorem{claim}{Claim}
\newtheorem{theorem}{Theorem}[section]
\newtheorem{conjecture}[theorem]{Conjecture}
\newtheorem{lemma}[theorem]{Lemma}
\newtheorem{corollary}[theorem]{Corollary}
\newtheorem{proposition}[theorem]{Proposition}
\newtheorem{question}[theorem]{Question}
\newtheorem{conj}[theorem]{Conjecture}
\newtheorem{defn}[theorem]{Definition}
\theoremstyle{definition}
\newtheorem{rmk}[theorem]{Remark}

\numberwithin{equation}{section}

\title{Gromov-Hausdorff limits of immortal K\"ahler-Ricci flows}

\author{Man-Chun Lee}
\address{Department of Mathematics, The Chinese University of Hong Kong, Shatin, N.T., Hong Kong}
\email{mclee@math.cuhk.edu.hk}
\author{Valentino Tosatti}
\address{Courant Institute School of Mathematics, Computing, and Data Science, New York University, 251 Mercer St, New York, NY 10012}
\email{tosatti@cims.nyu.edu}
\author{Junsheng Zhang}
\address{Courant Institute School of Mathematics, Computing, and Data Science, New York University, 251 Mercer St, New York, NY 10012}
\email{jz7561@nyu.edu}

\begin{abstract} We show that the normalized K\"ahler-Ricci flow on a compact K\"ahler manifold with semiample canonical bundle converges in the Gromov-Hausdorff topology to the metric completion of the twisted K\"ahler-Einstein metric on the canonical model, as conjectured by Song-Tian's analytic minimal model program.
\end{abstract}
\maketitle

\section{Introduction}
Let $(X^n,\omega_0)$ be a compact K\"ahler manifold, and let $\omega(t)$ be the solution of the normalized K\"ahler-Ricci flow
\begin{equation}\label{eq--KRF}
\left\{
                \begin{aligned}
                  &\frac{\de}{\de t}\omega(t)=-\Ric(\omega(t))-\omega(t),\\
                  &\omega(0)=\omega_0,
                \end{aligned}
              \right.
\end{equation}
starting at $\omega_0$. We are interested in the case when the solution $\omega(t)$ is {\em immortal}, i.e. it exists for all $t\geq 0$. By a result of Tian-Zhang \cite{TiZh}, this happens if and only if the canonical bundle $K_X$ is nef, i.e. $c_1(K_X)$ is a limit of K\"ahler classes. We are interested in the behavior of the metrics $\omega(t)$ as $t\to\infty$, and refer to the second-named author's survey on this topic \cite{To}.

The Abundance Conjecture in birational geometry, and its natural extension to K\"ahler manifolds, predicts that if $K_X$ is nef then $K_X$ is semiample, which means that $\ell K_X$ is base-point free for some $\ell\geq 1$. Abundance is currently known for $n\leq 3$, and we will assume from now on that $K_X$ is indeed semiample.

In this case, the Kodaira dimension $m:=\kappa(X)$ satisfies $0\leq m\leq n$, and the linear system $|\ell K_X|$ for $\ell$ sufficiently divisible gives a fiber space (the Iitaka fibration of $X$)
\begin{equation}
f:X\to Y\subset\mathbb{P}^N \text{ with } \ell K_X=f^*(\mathcal{O}_{\mathbb{P}^N}(1)|_Y),
\end{equation}
 onto a normal projective variety $Y$ of dimension $\dim Y=m$. This implies that the generic fibers of $f$ are Calabi-Yau $(n-m)$-folds. The variety $Y=\mathrm{Proj}\oplus_{k\geq 0}H^0(X,kK_X)$ is the canonical model of $X$.

We define $D\subset Y$ as the closed subvariety given by the union of the singularities of $Y$ together with the discriminant locus of $f$, so that if we call $X^\circ:=X\backslash f^{-1}(D)$ and $Y^\circ:=Y\backslash D$, then $f:X^\circ\to Y^\circ$ is a proper holomorphic submersion with connected $(n-m)$-dimensional Calabi-Yau fibers.

In the case when $m=0$, we have that $Y$ is a point, $D=\emptyset$, $X$ itself is Calabi-Yau, and the behavior of the flow has been fully understood since the 1980s \cite{Cao}: the flow \eqref{eq--KRF} shrinks the metric smoothly to zero as $t\to\infty$ (in particular, $(X,\omega(t))$ converges to a point in the Gromov-Hausdorff topology), and reparametrizing the flow to have constant volume, it converges smoothly to a Ricci-flat K\"ahler metric on $X$.

The case $m=n$ is also special, because in this case the flow is volume non-collapsed, in the sense that $\vol(X,\omega(t))\geq c>0$ for all $t\geq 0$. In this case $K_X$ is nef and big, and $X$ is known as a minimal model of general type. In this case, Tsuji \cite{Ts} and Tian-Zhang \cite{TiZh} proved that the flow converges smoothly on compact subsets of $X^\circ$ to a negative K\"ahler-Einstein metric $\omega_{\rm can}$ on $Y^\circ$ (pulled back via $f:X^\circ\to Y^\circ,$ which is an isomorphism in this case).

The remaining cases when $0<m<n$ turned out to be vastly more complicated, due to the fact that the flow is now volume collapsing as $t\to\infty$. This setup was first considered by Song-Tian in \cite{ST0,ST}. In this case, after much work by many people (see \cite{To} for references), it was recently shown by Hein and the first and second-named authors \cite{HLT} that as $t\to\infty$, we have $\omega(t)\to f^*\omega_{\rm can}$ smoothly on $X^\circ$, where $\omega_{\can}$ is a twisted K\"ahler-Einstein metric on $Y^\circ$, constructed by Song-Tian \cite{ST}.

In this paper we are interested in the global limiting behavior of $(X,\omega(t))$ as $t\to\infty$, in the Gromov-Hausdorff topology. The following basic conjecture is due to Song-Tian \cite[P.651]{ST0}, \cite[Conjectures 6.2 and 6.3]{ST3}, and it fits into their picture of the analytic minimal model program:
\begin{conjecture}\label{con}
Let $(X^n,\omega_0)$ be a compact K\"ahler manifold with $K_X$ semiample, and let $\omega(t)$ be the solution of the K\"ahler-Ricci flow \eqref{eq--KRF}. Then as $t\to\infty$, $(X,\omega(t))$ converges in the Gromov-Hausdorff topology to the metric completion of $(Y^\circ,\omega_{\rm can})$. This is a compact metric space homeomorphic to the canonical model $Y$.
\end{conjecture}

This conjecture has received much attention since it was first posed in 2006. As mentioned above, we now know \cite{HLT} that $\omega(t)$ converges locally smoothly to $f^*\omega_{\rm can}$ on $X^\circ$. However, obtaining reasonable estimates for the flow near $f^{-1}(D)=X\backslash X^\circ$ has proved to be very challenging.

The case $m=n$ was studied first, and the conjecture is known in low dimensions by Guo-Song-Weinkove \cite{GSW} and Tian-Zhang \cite{TiZ} or under the additional assumption of a uniform Ricci lower bound by Guo \cite{Gu}. When $m=n$, Conjecture \ref{con} was settled by Wang \cite{Wang}. The harder case when $0<m<n$ has also been much studied. In particular, Conjecture \ref{con} was proved by Song-Tian-Zhang \cite{STZ} when $m=1$ and the generic fibers of $f$ are tori, and by Li and the second-named author \cite{LT} in arbitrary dimensions assuming that $Y$ is smooth and the divisorial components of $D$ have simple normal crossings (which always holds when $m=1$).

Let us remark that the uniform diameter bound and the existence of subsequential Gromov-Hausdorff limits of $(X,\omega(t))$ were only achieved recently in \cite{JS}, and even more generally in \cite{GPSS, GPSS2,guedjT,vu2024} assuming only that $K_X$ is nef. Lastly, Sz\'ekelyhidi \cite{Sz} has very recently shown that in the setting of Conjecture \ref{con}, the metric completion of $(Y^\circ,\omega_{\rm can})$ is homeomorphic to $Y$, is a non-collapsed $\rcd(-1,2m)$-space, and inside this metric space, $Y\backslash Y^\circ$ has real Hausdorff codimension at least $2$. We will denote this metric completion by $(Y,d_{\rm can})$.

Our main result finally resolves Conjecture \ref{con} in general:

\begin{theorem}\label{main}In the setting of Conjecture \ref{con}, we have that as $t\to\infty$ the flow $(X,\omega(t))$ converges in the Gromov-Hausdorff topology to $(Y,d_{\mathrm{can}})$. In particular, Conjecture \ref{con} holds.
\end{theorem}

As a consequence of our result, together with the fact that $(Y,d_{\rm can})$ is a non-collapsed $\mathrm{RCD}$ space, it follows that the regular part of the Gromov--Hausdorff limit space has certain convexity properties \cite{CN2012,deng2015}, which is not a priori clear due to the fact that $(X,\omega(t))$ may not have a uniform Ricci lower bound.

In the case when $m=n$, our arguments are different from those in \cite{Wang}, so we obtain a new proof of that case. For general $m$, our proof follows a strategy introduced in \cite{LT}, where a key input is Perelman's monotonicity of the reduced volume, viewed as a parabolic analogue of the Bishop-Gromov volume comparison for Riemannian metrics with Ricci curvature bounded below. This allows for some control over how much time minimizing $\mathcal{L}$-geodesics spend in a neighborhood of the ``bad region'' $X\backslash X^\circ=f^{-1}(D)$. However to complete this strategy, \cite{LT} had to impose some smoothness assumptions on \((X,D)\), under which \(\omega_{\rm can}\) is quasi-isometric, on a generic region, to a \emph{conical K\"ahler metric} up to a small logarithmic error \cite{GTZ2}.

Since no such conical description is available for a general \(Y\), in this paper we refine the analysis in \cite{LT} and carry out a more careful stratification of the singular set \(D\).
We establish the existence of $\mathcal{L}$-geodesics that avoid a small neighborhood of $f^{-1}(Y\setminus Y^{\reg})$ and  almost avoid a small neighborhood of $f^{-1}(D\cap Y^{\reg})$ (or more precisely, they visit this neighborhood a controlled number of times for a controlled length of time). After carefully choosing appropriate neighborhoods of $f^{-1}(D)$ (here the RCD property of $(Y,d_{\rm can})$ is crucial), and combining this almost-avoidance property with a H\"older estimate for $d_{\can}$ on $Y^{\reg}$, we complete the proof of Theorem \ref{main}.

The paper is organized as follows. In Section \ref{sec--pre}, we collect various results from the literature concerning estimates for $(X,\omega(t))$ and $(Y^{\circ},\omega_{\can})$. In Section \ref{sec-holder}, we establish a H\"older estimate for the potential of $\omega_{\can}$ following \cite{GKSS}, and then derive a H\"older estimate for $d_{\can}$ on $Y^{\rm reg}$ using \cite{Li}. In Section \ref{sec--reduction}, following \cite{LT}, we reduce the main theorem to an estimate asserting that the $d_{\can}$-distance can be bounded above by the $\mathcal{L}$-distance.
Section \ref{sec--main} contains the core of the paper. We first prove a general almost-avoidance principle, roughly stating that for suitably small neighborhoods of $f^{-1}(D)$, a generic $\mathcal{L}$-geodesic cannot intersect such neighborhoods too many times. Then we carefully choose the small neighborhoods of $f^{-1}(D)$ to apply the almost-avoidance principle to, and we use the RCD property of $(Y,d_{\can})$ to establish an upper volume bound for these neighborhoods.
 Finally combining the almost-avoidance principle with the H\"older estimate obtained in Section \ref{sec-holder}, we complete the proof.

\subsection*{Acknowledgements} We are grateful to G\'abor Sz\'ekelyhidi for very useful discussions, and to Shouhei Honda, Wenshuai Jiang and Xiaochun Rong for email communications, to Yifan Guo and Lei Zhang for discussions about Poincar\'e inequalities on minimal surfaces, and to Jian Song and Jacob Sturm for pointing out a mistake in an earlier draft. The last-named author also thanks Song Sun and Keshu Zhou for helpful discussions. The first-named author is supported by Hong Kong RGC grants No. 14300623 and No. 14304225, and an Asian Young Scientist Fellowship. The second-named author was partially supported by NSF grant DMS-2404599.

\section{Preliminaries}\label{sec--pre}
In this section we collect some known results from the literature which will be used in our proof of Theorem \ref{main}. The setup is as given in the Introduction.

\subsection*{Notation}
Throughout the paper, $\Psi\left(\epsilon_1, \ldots, \epsilon_k \mid a_1, \ldots, a_\ell\right)$ will denote an $\mathbb{R}$-valued function of these parameters, such that whenever we fix the parameters $a_1,\dots,a_\ell,$ we have \begin{equation}\lim _{\epsilon_1, \ldots, \epsilon_k \rightarrow 0} \Psi=0.\end{equation}
\subsection{The twisted K\"ahler-Einstein metric $\omega_{\rm can}$} By assumption, $K_X$ is semiample, so we can choose $\ell\geq 1$ sufficiently divisible such that $\ell K_X$ is base-point free and the linear system $|\ell K_X|$ defines the Iitaka fibration $f:X\to Y\subset\mathbb{P}^N$ of $X$. Pulling back the coordinate functions on $\mathbb{P}^N$ we obtain a basis $\{s_i\}$ of $H^0(X,\ell K_X)$, and using these we can define a smooth positive volume form $\mathcal{M}$ on $X$ by
\begin{equation}
\mathcal{M}=\left((-1)^{\frac{\ell n^2}{2}}\sum_i s_i\wedge\ov{s_i}\right)^\frac{1}{\ell}.
\end{equation}
In the following, we will let $\omega_Y:=\omega_{\rm FS}|_Y$ be the restriction of the Fubini-Study metric on $\mathbb{P}^N$, which satisfies $[f^*\omega_Y]=\ell c_1(K_X)$. It is shown in \cite{ST} that there is a unique function $\vp\in C^0(Y)\cap C^\infty(Y^\circ)$ which is $\omega_Y$-psh and solves the complex Monge-Amp\`ere equation
\begin{equation}\label{makrf}
(\omega_Y+\ddbar\vp)^m=e^\vp f_*(\mathcal{M}),
\end{equation}
pointwise on $Y^\circ$ and also globally on $Y$ in the sense of pluripotential theory (continuity of $\vp$ is proved in \cite{CGZ,DZ,EGZ}). On $Y^\circ$ we have that
\begin{equation}
\omega_{\rm can}:=\omega_Y+\ddbar\vp
\end{equation} is a K\"ahler metric, which satisfies the twisted K\"ahler-Einstein equation
\begin{equation}
\Ric(\omega_{\rm can})=-\omega_{\rm can}+\omega_{\rm WP},
\end{equation}
where $\omega_{\rm WP}\geq 0$ is a Weil-Petersson form on $Y^\circ$ which measures the variation of the complex structure of the fibers of $f$. Of course, in the case when $m=n$, the fibers of $f$ over $Y^\circ$ are just points, and in this case $\omega_{\rm WP}\equiv 0$, so that $\omega_{\rm can}$ is a K\"ahler-Einstein metric in this case.

\subsection{Known results about $\omega(t)$}
Let $\omega(t), t\geq 0$, be the solution of the K\"ahler-Ricci flow \eqref{eq--KRF} on $X$. On the region $X^\circ\subset X$, the behavior of the flow is completely understood: after much work by a number of people, it was finally proved in \cite{HLT} that
\begin{equation}\label{conv}
\omega(t)\to f^*\omega_{\rm can},
\end{equation}
in the smooth topology (with respect to $\omega_0$) on compact subsets of $X^\circ$ and with locally uniformly bounded Ricci curvature. In this paper, we will only use the locally uniform convergence which was proved earlier in \cite{TWY}, and more precisely we will need the following result from \cite[Theorem 1.2 and p.685]{TWY}: given a compact subset $K\subset X^\circ$, on $K$ we have
\begin{equation}\label{useless}
\|\omega(t)-(f^*\omega_{\rm can}+e^{-t}\omega_{\rm SRF})\|_{C^0(K,\omega(t))}\to 0,
\end{equation}
where $\omega_{\rm SRF}$ is a ``semi-Ricci flat'' closed real $(1,1)$-form on $X^\circ$ which restricts to a Ricci-flat K\"ahler metric on every fiber of $f$ in $X^\circ$ (in the case when $m=n$, these fibers are points, and $\omega_{\rm SRF}\equiv 0$). Observe that given $K\Subset X^\circ$, we can find $T>0$ such that $f^*\omega_{\rm can}+e^{-t}\omega_{\rm SRF}$ is a K\"ahler metric on $K$ for all $t\geq T$, uniformly equivalent to $f^*\omega_Y+e^{-t}\omega_0$. From \eqref{useless} we immediately deduce that for $\hat \omega_t:=f^*\omega_{\rm can}+e^{-t}\omega_{\rm SRF}$,
\begin{equation}\label{twyconv}
(1-\Psi(t^{-1}|K))\hat \omega_t\leq \omega(t)\leq (1+\Psi(t^{-1}|K))\hat \omega_t.
\end{equation}

The above bounds hold on compact subsets of $X^\circ$. We also have the following bounds which hold on all of $X\times[0,\infty)$:
by \cite{ST2} (and the earlier \cite{Zh} when $m=n$), we know that the scalar curvature of $\omega(t)$ is uniformly bounded, i.e.
\begin{equation}\label{stscal}
\sup_X|R(\omega(t))|\leq C,
\end{equation}
for all $t\geq 0$, and also that the volume form of $\omega(t)$ satisfies
\begin{equation}\label{volscal}
C^{-1}e^{-(n-m)t}\omega_0^{n}\leq\omega(t)^{n}\leq Ce^{-(n-m)t}\omega_0^{n},
\end{equation}
on $X\times [0,\infty).$
By the ``parabolic Schwarz Lemma'' estimate \cite[Proposition 2.2]{ST2} (and also \cite[(3.4)]{TZ3} for the case when $Y$ is singular), we know that on $X\times [0,\infty),$
\begin{equation}\label{schwarz}
\omega(t)\geq C^{-1}f^*\omega_Y.
\end{equation}

We will also need the following diameter and volume non-collapsing estimates, which were recently established in \cite{GPSS,GPSS2,GPSS3,guedjT,vu2024}:
\begin{equation}
\diam(X,\omega(t))\leq C,\quad t\geq 0,
\end{equation}
\begin{equation}\label{thm--volume non-collapsing}
\operatorname{Vol}_{\omega(t)}\left(B_{\omega(t)}(x, r)\right) \geq C_\delta^{-1} r^{2 n+\delta} \vol(X,\omega(t)),
\end{equation}
for any $0<\delta<1$ and $x\in X, 0<r<\diam(X,\omega(t)).$ Note that when $m=n$, one can even take $\delta=0$ by \cite{Wang}, although we will not need this; all we will use is that $\frac{\operatorname{Vol}_{\omega(t)}\left(B_{\omega(t)}(x, r)\right)}{\vol(X,\omega(t))}\geq F(r)>0$ for all $x\in X$ and $0<r<1$.

\subsection{Known results about $\omega_{\rm can}$}

As for the structure of $\omega_{\can}$, the following theorem was recently established in \cite[Theorem 15]{Sz} for general $0<m\leq n$. The case $m=n$, i.e., the volume non-collapsing case, was proved earlier in \cite[Theorem 1.2]{So}.
\begin{theorem}\label{thm--rcd}
	The metric completion of $(Y^{\circ}, \omega_{\can})$ is homeomorphic to $Y$ and is a non-collapsed $\rcd(-1,2m)$-space, which we will denote by $(Y, d_{\can})$. Moreover,
\begin{equation}
    \dim_{\mathcal H}(Y \setminus Y^{\circ}) \le 2m - 2.
\end{equation}
\end{theorem}
From the RCD property it follows in particular that there exists $\kappa>0$ such that for any $y\in Y$ and $0<r<1$, we have
			\begin{equation}\label{kazzo2}
				\mathcal H^{2m}(B^{d_{\can}}(y,r))\geq \kappa r^{2m},
			\end{equation}
and also that for any sequence of open sets $V_{\epsilon}$ of $Y$ such that $\bigcap_{\epsilon}V_{\epsilon}=D$, we have
			\begin{equation}\label{kazzo}
				\lim_{\epsilon\rightarrow 0}\mathcal H^{2m}(V_\epsilon)=0.
			\end{equation}
\begin{rmk}
When $m=n$, it is known that $(Y,d_{\rm can})$ is a non-collapsed Ricci limit space by \cite{So}. Moreover, in general combining \cite[Theorem 1.3]{STZ} and \cite[Theorem 4]{Sz}, we know that $(Y,d_{\rm can})$ is also a Ricci limit space. However, when $m<n$, we do not know whether $(Y,d_{\rm can})$ is a non-collapsed Ricci limit space in general, although this is true when $Y$ is smooth, see e.g. \cite[Proposition 2.1]{LT}.
\end{rmk}

\section{H\"older bound on the potential and distance function of $\omega_{\rm can}$}\label{sec-holder}
In this section we prove H\"older bounds for the K\"ahler potential of $\omega_{\rm can}$ and for the distance function $d_{\rm can}$ on $Y^{\rm reg}$. The proof of the potential H\"older bound follows closely the recent work of Guo-Ko\l odziej-Song-Sturm \cite{GKSS}, while the argument to deduce from this the H\"older bound for $d_{\rm can}$ follows the method discovered by Li \cite{Li}. In \cite{GKSS} the potential H\"older bound was proved for families of polarized K\"ahler manifolds and for smoothable projective varieties, under natural assumptions. Our observation is that using the RCD property of $(Y,d_{\rm can})$, we can run their arguments with only small changes even though in our case $Y$ is singular and need not be smoothable. For the reader's convenience, we provide most of the details.

\subsection{H\"older bound for the potential of $\omega_{\rm can}$}
Recall that on $Y^\circ$ we have the K\"ahler metric $\omega_{\rm can}=\omega_Y+\ddbar\vp$, which solves the complex Monge-Amp\`ere equation in \eqref{makrf}, and $\vp$ extends to a continuous function on $Y$ with $\sup_Y \varphi = 0$. Since the right-hand side of \eqref{makrf} is known to be in $L^p(Y,\omega_Y^m)$ for some $p>1$ (see \cite{ST}), if $Y$ were smooth then \cite{Kol} would imply that $\vp$ is H\"older continuous on $Y$. It is widely believed that this statement still holds for solutions of such Monge-Amp\`ere equations on normal compact K\"ahler analytic spaces. The main result of this subsection is that this does indeed hold for our $\omega_{\rm can}$:
\begin{theorem}\label{thm-holder continuity of potential}
The function $\vp$ on $Y$ is H\"older continuous (with respect to the distance function $d_Y$), i.e. we have
\begin{equation}\label{patetico}
    |\varphi(x) - \varphi(y)|\le C\, d_Y(x,y)^{\alpha},
\end{equation}
for some constants $C>0$ and $0<\alpha<1$ and for all $x,y \in Y$.
\end{theorem}

Here $d_Y$ denotes the associated distance function induced from $\omega_Y$. For ease of notation, we will denote by $L=\mathcal{O}_{\mathbb{P}^N}(1)|_Y$, and up to replacing it by a suitable positive multiple (without relabeling), we may assume that for any $k,\ell\in \mathbb N_{>0}$,
\begin{equation}\label{eq--surjectivity of multiplicative map}
	\Sym^{k}H^0(Y,\ell L)\rightarrow H^0(Y,k\ell L) \text{ is surjective.}
\end{equation}
Recall that $\omega_Y=\omega_{\rm FS}|_Y$, and let also $h_Y=h_{\rm FS}|_Y$ and define a singular Hermitian metric on $L$ by $h_L:=h_Ye^{-\varphi}$, whose curvature form is $\omega_{\rm can}$. Then we obtain $L^2$-Hermitian metrics $\hilb_k$ on global holomorphic sections $H^0(Y,kL)$ for any $k\in \mathbb N_{>0}$ by
\begin{equation}
	\langle s,\sigma\rangle_{\hilb_k}:=\int_Y\langle s,\sigma\rangle_{h_L^k}(k\omega_{\can})^m.
\end{equation}The density of states function is given by \begin{equation}
\rho_{k}=\sum_{\alpha}\vert s_{\alpha}\vert^2_{h_L^k},
\end{equation}where $\{s_{\alpha}\}$ is any $L^2$-orthonormal basis of $H^0(Y,kL)$ with respect to $\hilb_k$.

With these preparations, we have the following uniform bounds for the density of states function:
\begin{theorem}\label{thm--uniform bound on bergman kernel}
	There exists an $r\in \mathbb N_{>0}$ and $A>0$ such that for any $k\in \mathbb N_{>0}$ and $x\in Y$, we have
	\begin{equation}
		A^{-1}\leq \rho_{rk}(x)\leq A.
	\end{equation}
\end{theorem}
\begin{proof}
The upper bound for the density of states is standard in the smooth setting and follows from Moser iteration.
In the present singular setting, it can be proved using the $\mathrm{RCD}$ property of $(Y,\omega_{\can})$, see \cite[Proposition 19]{Sz2}.

The lower bound is essentially contained in \cite[Proposition 3.1]{LiuSz}, combined with an observation in \cite[Theorem 1.5]{ZhK}. As discussed in \cite[p.11]{Sz}, the argument in \cite[Proposition 5.1]{LiuSz}, which relies on \cite{CDS2}, implies that any metric cone $V$ that arises as a pointed Gromov-Hausdorff limit of $(Y,k_i\omega_i,x_i)$ for some $k_i\rightarrow \infty$ and points $x_i\in Y$, satisfies that $V\setminus \mathcal R_{\epsilon}(V)$ has zero $2$-capacity for any $\ve>0$, where $\mathcal{R}_\ve$ is defined in analogy with \eqref{reg} below.
Therefore \cite[Proposition 3.1]{LiuSz} applies to the spaces $\{(Y, k\omega_{\can})\mid k\in \mathbb N_{>0}\}$, and it shows that there exist $K\in \mathbb N_{>0}$ and $b>0$ such that for any $k\in \mathbb N_{>0}$ and $x\in  Y$, there exists an integer $\ell=\ell(x,k)\in [1,K]$ such that
\begin{equation}\label{eq--lower bound changing polarization}
	\rho_{\ell,kL}(x)\geq b.
\end{equation}Here we use $\rho_{\ell,kL}$ to emphasize that the density of states function for the line bundle $kL$, which clearly satisfies
\begin{equation}\label{scaling of polarization}
	\rho_{\ell,kL}=\rho_{\ell k,L}.
\end{equation} If we let $r=K!$, then the lower bound for $\rho_{rk,L}$ follows from \eqref{eq--lower bound changing polarization}, \eqref{scaling of polarization} and the rough lower bound in \cite[Lemma 3.1]{DS1} for the density of states function when raising powers.
\end{proof}

Replacing $L$ by $rL$, we may assume in the following that $r=1$. Since $Y$ is a projective variety, we know that $Y^{\reg}$ is an essentially Stein manifold, i.e. there exists an analytic hypersurface $V\subset Y^{\reg}$ such that $Y^{\reg}\setminus V$ is Stein. To see this, we can choose an effective ample divisor $D$ such that $Y^{\sing}\subset \supp (D)$ and let $V=\supp (D)\setminus Y^{\sing}$. Then
\begin{equation}
	Y^{\reg}\setminus V=Y\setminus \supp(D)
\end{equation}is Stein. By \cite[Theorem 2.1, Theorem 2.2]{varolin}, we know that the Skoda division theorem still holds on $Y^{\reg}$ and since $Y$ is normal, any holomorphic section in $H^0(Y^{\reg},M)$ where $M$ is a line bundle on $Y$ extends to a section in $H^0(Y,M)$. Then one can argue exactly as in \cite[Theorem 5]{GKSS} (see also \cite[Proposition 7]{LiT}, and see also the related work in \cite{Fi}), using the bound on the density of states function to get the following:
\begin{theorem}\label{thm--effective generation}
	Let $\ell_0,\ell_1\geq 1$ and $k\geq (m+2+\ell_1)\ell_0$. Let $s_0,\cdots,s_{N_{\ell_0}}$ be an orthonormal basis of $H^0(Y,\ell_0L)$ with respect to $\hilb_{\ell_0}$. Then for any $U\in H^0(Y,kL)$, we can write
	\begin{equation}
		U=\sum_{\alpha_1, \ldots, \alpha_{\ell_1}=0}^{N_{\ell_0}} U\left(\alpha_1, \ldots, \alpha_{\ell_1}\right) s_{\alpha_1} \cdots s_{\alpha_{\ell_1}},
	\end{equation}with the estimate
	\begin{equation}
		\|U(\alpha_1,\cdots,\alpha_{\ell_1})\|^2_{\hilb_{k-\ell_1\ell_0}}\leq \frac{(m+\ell_1)!}{m!\ell_1!}\frac{(k-\ell_1\ell_0)^m}{k^m}A^{2m+2+\ell_1}\|U\|^2_{\hilb_k}.
	\end{equation}
\end{theorem}

Let $\{s_0,\cdots, s_{N_1}\}$ be an orthonormal basis of $H^0(Y,L)$ with respect to $\hilb_1$, where $N_1=N$.
By increasing $A$ if necessary, we have that
\begin{equation}\label{eq--bound on potential}
    -A \leq \varphi \leq 0,
\end{equation}
on $Y$.
Then, as in \cite{GKSS}, we can define the approximating potentials
\begin{equation}
	\varphi_k=\frac{1}{k} \log \left(\sum_{i=0}^{N_k}\left|\sigma_i\right|_{h_Y^k}^2\right)=\frac{1}{k} \log \left(\sum_{i=0}^{N_k}\left|\sigma_i\right|_{h_L^k}^2\right)+\varphi,
\end{equation}where $\{\sigma_i\}$ is an orthonormal basis of $H^0(Y,kL)$ with respect to $\hilb_k$. Thanks to the uniform bound for the density of states function in Theorem \ref{thm--uniform bound on bergman kernel}, we know that
\begin{equation}\label{eq--C^0-approximation}
	\begin{aligned}
		\|\varphi_k-\varphi\|_{L^{\infty}(Y)}\leq \frac{\log A}{k}.
	\end{aligned}
\end{equation}

In the following, all connections on $Y^{\rm reg}$ are taken with respect to $\omega_Y$ and $h_Y$.
\begin{lemma}\label{lem--exponential}
	There exists a $B>0$ such that for any $k\in \mathbb N_{>0}$ and  $\sigma\in H^0(Y,kL)$, and $x\in Y^{\reg}$ we have
	\begin{equation}
		\|\nabla \sigma\|_{L^{\infty}(Y^{\reg},\omega_Y,h_Y^k)}\leq B^k \|\sigma\|_{\hilb_k}.
	\end{equation}
\end{lemma}

\begin{proof}
	After rescaling, we can assume  $\|\sigma\|_{\hilb_k}=1$. When $k=1,$ after a unitary transformation we can assume that $\sigma$ is the restriction of $E_0$, the first coordinate function on $\mathbb C^{N_1+1}$. Then we have that for any $x\in Y^{\reg}$,
	\begin{equation}\label{eq-trivial inequality}
		|\nabla \sigma|_{\omega_Y,h_Y}(x)\leq \|\nabla E_0\|_{L^{\infty}(\mathbb P^{N_1},\omega_{\rm FS},h_{\rm FS})}\leq C_1.
	\end{equation}For general $k$,
since we have the surjective map  \eqref{eq--surjectivity of multiplicative map}, we can assume that $\sigma$ is the restriction of a homogeneous polynomial $p=p(E_0,\cdots,E_{N_1})$. Since any two norms on a finite dimensional vector space are equivalent, the coefficients of $p$ are bounded by a uniform constant $C_k$. Then a similar argument as in \eqref{eq-trivial inequality} gives
	\begin{equation}\label{eq--non-effective bound}
		\|\nabla \sigma\|_{L^{\infty}(Y^{\reg},\omega_Y,h_Y^k)}\leq C_k=C_k \|\sigma\|_{\hilb_k}.
	\end{equation}
	We thus need to show that we can take $C_k=B^k$. Clearly we may assume without loss that $k>m+2$.
	Recall that we have fixed $\{s_0,\cdots, s_{N_1}\}$, an orthonormal basis of $H^0(Y,L)$ with respect to $\hilb_1$. By Theorem \ref{thm--uniform bound on bergman kernel} and the normalization $\sup_Y \varphi=0$, we know that
	\begin{equation}\label{eq--C0bound for potential}
		\|s_i\|_{L^{\infty}(Y^{\reg},h_Y)}\leq \|s_i\|_{L^{\infty}(Y^{\reg},h_L)}\leq A, \quad 0\leq i\leq N_1.
	\end{equation}
	We then apply Theorem \ref{thm--effective generation} with
	\begin{equation}
		\ell_0=1, \quad \ell_1=k-m-2,
	\end{equation}
	to get that
	\begin{equation}\label{eq--effective generation}
		\sigma=\sum \sigma(\alpha_1,\cdots \alpha_{k-m-2})s_{\alpha_1}\cdots
		s_{\alpha_{k-m-2}},
	\end{equation}with $\sigma(\alpha_1,\cdots \alpha_{k-m-2})\in H^0(Y,(m+2)L)$ and
	\begin{equation}
	\begin{aligned}
	\|\sigma(\alpha_1,\cdots \alpha_{k-m-2})\|^2_{\hilb_{m+2}}&\leq \frac{(k-2)!}{m!(k-m-2)!}\frac{(m+2)^m}{k^m}A^{k+m}\\
		&\leq C(m,A)A^k
	\end{aligned}\end{equation}
and then Theorem \ref{thm--uniform bound on bergman kernel} gives
\begin{equation}\label{sux}
	\|\sigma(\alpha_1,\cdots \alpha_{k-m-2})\|^2_{L^\infty(Y^{\rm reg}, \omega_Y, h_Y^{m+2})}\leq C(m,A)A^{k+1}.
	\end{equation}
Then by the non-effective bound \eqref{eq--non-effective bound}, we know that
	\begin{equation}\label{eq--noneff bound}
		|\nabla \sigma(\alpha_1,\cdots \alpha_{k-m-2})|\leq C(m,A)A^{k/2}.
	\end{equation}
The key point here is that the constant $C(m,A)$ is independent of $k$.
	Then using \eqref{eq-trivial inequality}, \eqref{eq--C0bound for potential}, \eqref{eq--effective generation}, \eqref{sux} and \eqref{eq--noneff bound}, and being wasteful with the powers of $k$, we can estimate
	\begin{equation}
	\begin{aligned}
&\|\nabla \sigma\|_{L^{\infty}(Y^{\reg},\omega_Y,h_Y^k)}\\
		&\leq C(m,A)N_1^{k-m-2}(A^k\cdot A^{k-m-2} +A^{k+1}\cdot A^{k-m-2}\cdot C_1(k-m-2)) \\
		&\leq C(m,A)N_1^k(A+1)^{2k}\leq B^k,
	\end{aligned}	
	\end{equation}by choosing $B$ large.	
\end{proof}

\begin{lemma}\label{lem--polynomial bound}
	There exists $K>0$ such that for any $r\in \mathbb N_{\geq 2}$,  $k=(m+2)^r$ and $U\in H^0(Y, kL)$, we have
	\begin{equation}
			\|\nabla U\|_{L^{\infty}(Y^{\reg},\omega_Y,h_Y^k)}\leq k^K \|U\|_{\hilb_k}.
	\end{equation}
\end{lemma}

\begin{proof}
	For $r\geq 2$,  we take $\ell_0=(m+2)^{r-2}$ and $\ell_1=m^2+3m+2$ and $k=(m+2)^r$, which satisfies
	\begin{equation}
		k=(m+2+\ell_1)\ell_0.
	\end{equation} Let $s_0,\cdots,s_{N_{\ell_0}}$ be an orthonormal basis of $H^0(Y,\ell_0L)$ with respect to $\hilb_{\ell_0}$. For a given point $x \in Y^{\reg}$, by considering the kernel of the evaluation map at $x$ and the $1$-jet map at $x$, together with their orthogonal complements, we may choose the basis $\{s_i\}$ such that $s_0(x) \neq 0$, the sections $\{s_i \mid 1 \leq i \leq m\}$ vanish at $x$ to first order, and the remaining sections vanish at $x$ to at least second order. We can then apply Theorem \ref{thm--effective generation}, and write
\begin{equation}
U=\sum_{j=0}^mU_js_0^{\ell_1-1}s_j,
\end{equation}
and so
	\begin{equation}\label{eq--gradient}
		\nabla U(x)=\nabla \left(\sum_{j=0}^m U_js_0^{\ell_1-1}s_j\right)(x)
	\end{equation}with
	\begin{equation}\label{eq--estimate for uj}
	\begin{aligned}
				\|U_j\|_{\hilb_{(m+2)^{r-1}}}&\leq \left(\frac{(m^2+4m+2)!}{m!(m^2+3m+2)!} A^{m^2+5m+4}\right)^{\frac12} \|U\|_{\hilb_{(m+2)^r}}\\
				& =:C(m,A)\|U\|_{\hilb_{(m+2)^r}}.
	\end{aligned}
	\end{equation}Expanding \eqref{eq--gradient}, we get that at the point $x$, we have
	\begin{equation}
		\nabla U=s_0^{\ell_1}\nabla U_0+U_0\ell_1s_0^{\ell_1-1}\nabla s_0+\sum_{j=1}^m U_js_0^{\ell_1-1}\nabla s_j.
	\end{equation}For simplicity of notation in the following, we omit the metrics, but it will be understood that all norms are taken with respect to $\omega_Y$ and $h_Y$, unless specified otherwise.
We can estimate
\begin{equation}\label{eq-precise term bound}
	\begin{aligned}
		|\nabla U|(x)&\leq A^{\ell_1}|\nabla U_0|+C(m,A)(\ell_1+1)A^{\ell_1-1}\sum_{j=0}^m|\nabla s_j|\|U_j\|_{\hilb_{(m+2)^r}}.
	\end{aligned}
\end{equation}
Then combining  Lemma \ref{lem--exponential} with \eqref{eq--estimate for uj}, we get that
\begin{equation}
	|\nabla U|(x)\leq DB^{(m+2)^{r-1}}\|U\|_{\hilb_{(m+2)^r}},
\end{equation}where the constant
\begin{equation}
D=D(m,A)=C(m,A)A^{\ell_1}(1+m(\ell_1+1)A^{-1})\geq 1.
\end{equation}

Then we prove by induction on $t$ that for any $1\leq t\leq r-2$ and $U\in H^0(Y, (m+2)^rL)$, we have
\begin{equation}\label{eq--claim}
	\|\nabla U\|_{L^{\infty}(Y^{\reg},\omega_Y,h_Y^{(m+2)^r})}\leq  D^tB^{(m+2)^{r-t}}\|U\|_{\hilb_{(m+2)^r}}.
\end{equation}
The case $t=1$ was just proved by the previous argument. We assume the claim is true for a given $t\geq 1$, and then prove it holds for $t+1$. We consider an $r$ with $r-2\geq t+1$ and $U\in H^0(Y, (m+2)^rL)$. We have \eqref{eq--gradient} and \eqref{eq--estimate for uj} as before with $U_j\in H^0(Y, (m+2)^{r-1}L)$ and with the normalization
\begin{equation}
	\|U\|_{\hilb_{(m+2)^r}}=1.
\end{equation}Since $r-1-2\geq t$, we can apply the inductive assumption to $U_j$ to get
\begin{equation}\label{eq-after induction assumption on uj}
	\|\nabla U_j\|_{L^{\infty}(Y^{\reg},\omega_Y,h_Y^{(m+2)^{r-1}})}\leq D^t B^{(m+2)^{r-1-t}} \|U_j\|_{\hilb_{(m+2)^{r-1}}}.
\end{equation}Similarly we have $s_j\in H^0(Y, (m+2)^{r-2})$, $r-2-2\geq t-1$ and hence by the inductive assumption we obtain that
\begin{equation}\label{eq-after induction assumption on sj}
	\|\nabla s_j\|_{L^{\infty}(Y^{\reg},\omega_Y,h_Y^{(m+2)^{r-2}})}\leq D^{t-1} B^{(m+2)^{r-1-t}} \|s_j\|_{\hilb_{(m+2)^{r-2}}}.
\end{equation}
	Then combining \eqref{eq-precise term bound}, \eqref{eq-after induction assumption on uj} and \eqref{eq-after induction assumption on sj}, we obtain that at any given point $x\in Y^{\reg}$, we have
	\begin{equation}
	\begin{aligned}
		|\nabla U|(x)\leq&  C(m,A)\left(
A^{\ell_1}D^t B^{(m+2)^{r-1-t}}+m(\ell_1+1)A^{\ell_1-1}D^{t-1} B^{(m+2)^{r-1-t}}\right)\\
		\leq & C(m,A)A^{\ell_1}(1+A^{-1}D^{-1}m(\ell_1+1))D^{t}B^{(m+2)^{r-1-t}}\\
		\leq& D^{t+1}B^{(m+2)^{r-1-t}}.
	\end{aligned}
	\end{equation} This proves \eqref{eq--claim}.
	Taking $t=r-2$, we obtain the desired inequality stated in the Lemma.
	\end{proof}

\begin{proof}[Proof of Theorem \ref{thm-holder continuity of potential}] We take $k=(m+2)^r$ for some $r\geq 2$ to be specified below. As a consequence of Lemma \ref{lem--polynomial bound} and the definition of $\varphi_k$, we obtain that
\begin{equation}\label{puerile}
	\|\nabla \varphi_k\|_{L^{\infty}(Y^{\reg})}\leq Ak^{K-1}.
\end{equation}
Using \eqref{eq--C^0-approximation} and \eqref{puerile}, given any two distinct points $x,y\in Y^{\reg}$ we can then estimate
\begin{equation}
\begin{aligned}
	\left|\varphi(x)-\varphi(y)\right| &\leq\left|\varphi(x)-\varphi_k(x)\right|+\left|\varphi_k(x)-\varphi_k(y)\right|+\left|\varphi_k(y)-\varphi(y)\right|\\
	&\leq \frac{2\log A}{k}+A k^{K-1}d_Y(x,y).
\end{aligned}
	\end{equation}Choosing
	\begin{equation}
		r=\left\lfloor\frac{\log \left(d_Y(x,y)^{-\frac{1}{2K}}\right)}{\log(m+2)}\right\rfloor,
	\end{equation}
then gives
	\begin{equation}
		|\varphi(x)-\varphi(y)|\leq C_0\left(d_Y(x,y)^{\frac{1}{2K}}+ d_Y(x,y)^{\frac{1}{2}+\frac{1}{2K}}\right),
	\end{equation}
for some uniform constant $C_0$, which proves \eqref{patetico}.
\end{proof}

\subsection{H\"older bound for $d_{\rm can}$} In this subsection, we use the H\"older bound in Theorem \ref{thm-holder continuity of potential} and the method of Li \cite{Li} to prove a H\"older bound for $d_{\rm can}$ away from the singular locus of $Y$.

Recall that we have an embedding $Y\hookrightarrow\mathbb{P}^N$ and we have defined $\omega_Y:=\omega_{\rm FS}|_Y$, and $d_Y$ is the intrinsic distance function on $Y$ defined by the metric $\omega_Y$. There is also an ``extrinsic'' distance function $d_{\rm ext}$ on $Y$, which is defined by restricting to $Y$ the distance function $d_{\rm FS}$ of the Fubini-Study metric on $\mathbb{P}^N$. Clearly we have
\begin{equation}\label{easy lower bound}
	d_{\rm ext}\leq d_Y.
\end{equation}It follows from \cite[Theorem B, Proposition 1, Chapter II]{stol} that there is $C>0$ such that for all $0<r<1$ and all $x\in Y$ we have
\begin{equation}\label{volume ratio two side bound}
C^{-1}r^{2m}\leq \vol(B^{d_{\rm ext}}(x,r), \omega_Y^m)\leq Cr^{2m}.
\end{equation}
Moreover as a simple consequence of the \L ojasiewicz inequality one has
\begin{equation}
d_Y\leq Cd_{\rm ext}^\alpha,
\end{equation}
for some $C,\alpha>0$, see \cite[Proposition 4.6]{GGZ}. Thanks to this and Theorem \ref{thm-holder continuity of potential}, we see that $\vp$ is also H\"older continuous with respect to $d_{\mathrm{ext}}$, with H\"older exponent that we still denote by $\alpha$ (up to shrinking it).

For any \(r>0\), set
\[
	Y_r:=\{y\in Y \mid d_Y(y,Y\setminus Y^{\reg})>r\}.
\]
Then, for any \(x\in Y_r\) and any \(s\in (0,r/2)\), the \(d_Y\)-ball
\(B^{d_Y}(x,s)\) is contained in \(Y_{r/2}\). Moreover, on \(Y_{r/2}\), this ball
coincides with the geodesic ball defined by the smooth Riemannian metric
\(\omega_Y\) on $Y_{r/2}$. Thus, in what follows, for \(x\in Y_r\) and \(s\in (0,r/2)\), we
continue to write \(B(x,s)\) for \(\omega_Y|_{Y_{r/2}}\)-geodesic ball centered at \(x\)
with radius \(s\).

\begin{theorem}\label{zatta}
	For any $r>0$ there exists $C=C(r)>0$ such that for all $x,y\in Y_r$, we have
	\begin{equation}\label{holder}
		d_{\can}(x,y) \le C\, d_{Y}(x,y)^{\frac{\alpha}{2}},
	\end{equation}
 where $\alpha>0$ is the H\"older continuity index of $\varphi$ with respect to $d_{\rm ext}$.
\end{theorem}
\begin{proof}
This is a direct modification of the argument of Li \cite[Theorem 4.1]{Li}. For the reader's convenience, we provide the details.

Since $(Y_{r/2},\omega_Y)$ is a smooth manifold with bounded geometry (depending on $r$), we can find constants $C=C(r)>0, s_0=s_0(r)\in (0,\frac{r}{2})$ such that the volume bound
\begin{equation}\label{satana}
C^{-1}s^{2m}\leq \vol(B(x,s), \omega_Y^m)\leq Cs^{2m},
\end{equation}
and the Sobolev-Poincar\'e inequality
\begin{equation}\label{sobo}
	\left(\int_{B(x,s)}|u-u_{x,s}|^{\frac{2m}{2m-1}}\omega_Y^m\right)^{\frac{2m-1}{2m}}\leq C\int_{B(x,Cs)}|\nabla u|_{g_Y}\omega_Y^{m},
\end{equation}
hold for all $x\in Y_r$ and $0<s<s_0$, where
\begin{equation}
u_{x,s}:=\frac{1}{\vol(B(x,s), \omega_Y^m)}\int_{B(x,s)}u\omega_Y^m.
\end{equation}

For each fixed $x_0\in Y_r$, we denote $u(y):=d_{\rm can}(x_0,y)$ so that
	\begin{equation}\label{satan}
		|\nabla u|^2_{g_Y}\leq \tr{\omega_Y}{\omega_{\rm can}}.
	\end{equation}
Then, for $x\in Y_r$ and $s<\frac{s_0}{2}$, restricting to $Y$ a cutoff function on $\mathbb{P}^N$, we obtain a smooth function $\chi$ on $Y$ which is supported on $B^{d_{\rm ext}}(x,2s)$ and identically equal to $1$
on $B^{d_{\rm ext}}(x,s)$, such that
\begin{equation}\label{satanasso}
	\ii\de\db \chi\leq \frac{10}{s^2}\omega_Y.
\end{equation}
Recalling that $\omega_{\rm can}=\omega_Y+\ddbar\vp$, integrating by parts, and using \eqref{easy lower bound}, \eqref{volume ratio two side bound},\eqref{satana}, \eqref{satan},  \eqref{satanasso} and the $d_{\rm ext}$-H\"older bound for $\vp$, we get that
\begin{equation}\label{zatan}
	\begin{aligned}
\int_{B(x,s)} |\nabla u|^2_{g_Y}\omega_Y^m& \leq \int_{B(x,s)}\left(\tr{\omega_Y} \omega_{\rm can}\right) \omega_Y^m\leq m \int_{B^{d_{\rm ext}}(x,s)} \omega_{\rm can}\wedge \omega_Y^{m-1}\\
& \leq m \int_{Y} \chi \omega_{\rm can} \wedge\omega_Y^{m-1} \\
& \leq m \vol(B^{d_{\rm ext}}(x,2s), \omega_Y^m)+m \int_Y \chi \ii\de\db \varphi\wedge \omega_Y^{m-1} \\
& =C s^{2 m}+m \int_Y\left(\varphi-\inf _{B^{d_{\rm ext}}(x,2 s)} \varphi\right) \ii\de\db \chi \wedge \omega_Y^{m-1} \\
& \leq C s^{2 m}+C\left(\operatorname{osc}_{B^{d_{\rm ext}}(x,2 s)} \varphi\right) s^{-2}\vol(B(x,2s), \omega_Y^m)\\
& \leq C s^{2 m-2+\alpha}.
\end{aligned}
\end{equation}
Using \eqref{satana}, \eqref{zatan} and \eqref{sobo}, we can estimate
\begin{equation}\label{buddha}\begin{aligned}
|u_{x,s}-u_{x,2s}|&\leq \frac{1}{\vol(B(x,s), \omega_Y^m)}\int_{B(x,s)}|u-u_{x,2s}|\omega_Y^m\\
&\leq \frac{1}{\vol(B(x,s), \omega_Y^m)}\int_{B(x,2s)}|u-u_{x,2s}|\omega_Y^m\\
&\leq \frac{\vol(B(x,2s), \omega_Y^m)^{\frac{1}{2m}}}{\vol(B(x,s), \omega_Y^m)}\left(\int_{B(x,2s)}|u-u_{x,2s}|^{\frac{2m}{2m-1}}\omega_Y^m\right)^{\frac{2m-1}{2m}}\\
&\leq Cs^{1-2m}\int_{B(x,Cs)}|\nabla u|_{g_Y}\omega_Y^{m}\\
&\leq Cs^{1-m}\left(\int_{B(x,Cs)}|\nabla u|^2_{g_Y}\omega_Y^{m}\right)^{\frac{1}{2}}\leq Cs^{\frac{\alpha}{2}}.
\end{aligned}\end{equation}
For $i\geq 1$ let $u_i:=u_{x,2^{-i}s}$, so that \eqref{buddha} gives
\begin{equation}
|u_{i+1}-u_i|\leq Cs^{\alpha/2} 2^{-i\alpha/2},
\end{equation}
and since $\lim_{i\to\infty}u_i=u(x)$ (using that $u$ is continuous and $x\in Y^{\rm reg}$), we obtain
\begin{equation}
\begin{split}\label{eq--holder with average}
|u(x)-u_{x,s}|&\leq \sum_{i=1}^\infty |u_{i+1}-u_i|\\
&\leq Cs^{\alpha/2} \sum_{i=1}^\infty 2^{-i\alpha/2}=C's^{\alpha/2}.
\end{split}
\end{equation}

In order to prove the H\"older bound, since \(d_{\can}\) has bounded diameter and we allow the constant to depend on \(r\), it suffices to consider the case where \(x,y\in Y_r\),
\[
	d_Y(x,y)=:s/2\in (0,s_0/2),
\]
By \eqref{eq--holder with average},
we then obtain
\begin{equation}
\begin{split}
|u(x)-u(y)|&\leq |u(x)-u_{x,s}|+|u_{y,s}-u(y)|+|u_{x,s}-u_{y,s}|\\
&\leq Cs^{\alpha/2}+|u_{x,s}-u_{y,s}|,
\end{split}
\end{equation}
and to estimate the last term observe that $B(x,s/2)\subset B(x,s)\cap B(y,s)$, and use \eqref{satana} and \eqref{buddha} again to estimate
\begin{equation}
\begin{split}
|u_{x,s}-u_{y,s}|&\leq Cs^{-2m}\int_{B(x,s/2)}|u-u_{x,s}|\omega_Y^{m}+Cs^{-2m}\int_{B(x,s/2)}|u-u_{y,s}|\omega_Y^{m}\\
&\leq Cs^{-2m}\int_{B(x,s)}|u-u_{x,s}|\omega_Y^{m}+Cs^{-2m}\int_{B(y,s)}|u-u_{y,s}|\omega_Y^{m}\\
&\leq Cs^{\alpha/2},
\end{split}
\end{equation}
which shows that $|u(x)-u(y)|\leq Cd_{Y}(x,y)^{\alpha/2}$ for all $x,y \in Y_r$, and since $u(x_0)=0$, this proves \eqref{holder}.
\end{proof}

\begin{rmk}
In a previous version of this paper, we claimed a global H\"older bound for
\(d_{\can}\) on all of $Y$. However, as pointed out to us by Jian Song and Jacob Sturm, the proof contained a gap:
it relied on a Poincar\'e inequality for \(\omega_Y\) which does not appear to be currently
known. Our argument in Section \ref{sec--main} shows that this global H\"older bound for
\(d_{\can}\) is not needed for the main results of the paper. Nevertheless, it
remains an interesting question whether the required Poincar\'e inequality
holds, and whether the global H\"older bound for $d_{\rm can}$ can be established.
\end{rmk}

\section{Reduction to the key estimate}\label{sec--reduction}
In this section, we reduce the proof of Theorem \ref{main} to establishing a key estimate, stated in Proposition \ref{prop--main to prove}, which is a sharp upper bound for $d_{\can}$ in terms of Perelman's $\mathcal{L}$-distance. This reduction follows closely the arguments in \cite{LT}, with some caveats due to the fact that in \cite{LT} the space $Y$ was assumed to be smooth, while here we allow it to be singular. The proof of Proposition \ref{prop--main to prove} will then be given in the next section.

\subsection{Reduction to an upper for $d_{\can}$ by $d_t$}In the following, let $d_t$ denote the distance function on $X$ defined by the metric $\omega(t)$. To prove Theorem \ref{main} we need to show that $(X,d_t)$ converges in the Gromov-Hausdorff topology to $(Y,d_{\rm can})$. Recall that we have the analytic subvariety $D\subset Y$ so that $Y^\circ=Y\backslash D$ is smooth and $f$ is a submersion over $Y^\circ$. Given $\ve>0$ we define
\begin{equation}
	V_{\epsilon}:=\{y\in Y\mid d_{\can}(D,y)<\epsilon\}, \quad \tilde V_{\epsilon}:=f^{-1}(V_{\epsilon}).
\end{equation}and let $(Y\setminus V_{\epsilon},d_{\can})$ denote the restriction of the metric $d_{\can}$ from $Y$ to $Y\setminus V_{\epsilon}$, and similarly let $(X\setminus \tilde V_{\epsilon},d_t)$ be the restriction of $d_t$ from $X$ to $X\setminus \tilde V_{\epsilon}$.

We begin with the following simple observation:
\begin{lemma}\label{lem--gh close for smooth exhaustion}
	We have
	\begin{equation}\label{eq--gh close for Y}
		d_{\gh}((Y,d_{\can}),(Y\setminus V_{\epsilon},d_{\can}))=\Psi(\epsilon).
	\end{equation}
	\begin{equation}\label{eq--gh close for X}
		d_{\gh}((X,d_t),(X\setminus \tilde V_{\epsilon},d_t))=\Psi(\epsilon, t^{-1}).
	\end{equation}
\end{lemma}

\begin{proof}
	We want to show the natural inclusion map gives the desired Gromov-Hausdorff approximations.	
	For \eqref{eq--gh close for Y}, it is enough to show that $Y\setminus V_{\epsilon}$ is $\Psi(\epsilon)$-dense inside $(Y,d_{\can})$, i.e. 	\begin{equation}
		d_{\can}(y, \partial V_{\epsilon})=\Psi(\epsilon), \text{ for any } y\in V_\epsilon.
	\end{equation}This follows from the volume non-collapsing of $d_{\can}$ proved in \eqref{kazzo2}, and the fact that the volume of $V_{\epsilon}$ goes to 0, see \eqref{kazzo}.

	For \eqref{eq--gh close for X}, it is enough to show that $X\setminus \tilde V_{\epsilon}$ is $\Psi(\epsilon,t^{-1})$-dense inside $(X,d_t)$, i.e. 	\begin{equation}\label{close to boundary}
		d_{t}(x, \partial \tilde V_{\epsilon})=\Psi(\epsilon,t^{-1}), \text{ for any } x\in \tilde V_\epsilon.
	\end{equation}
Using \eqref{volscal}, \eqref{makrf} (and the boundedness of the function $\vp$ that appears there) we can then estimate
	\begin{equation}\label{fut}
		\frac{\vol(\tilde V_{\epsilon},\omega(t)^n)}{\vol(X,\omega(t)^n)}\leq C\int_{\ti{V}_\ve}\omega_0^n\leq C\int_{\ti{V}_\ve}\mathcal{M}=C\int_{V_\ve}f_*(\mathcal{M})\leq
C\int_{V_\ve}\omega_{\can}^m,
	\end{equation}
for a constant $C$ independent of $t$ and $\epsilon$.
	Then \eqref{close to boundary} follows from this volume upper bound for $\tilde V_{\epsilon}$ together with the volume non-collapsing estimate for $\omega(t)$ in \eqref{thm--volume non-collapsing}.
\end{proof}

The next lemma is also elementary:
\begin{lemma}\label{lem--convexity on y}
	Given $\epsilon,\delta>0$, there exists $\epsilon'=\epsilon'(\epsilon,\delta)$ such that for any $x, y\in Y\setminus V_{\epsilon}$, there exists a smooth curve $\gamma$ contained in $Y\setminus V_{\epsilon'}$, connecting $x$ and $y$ such that
	\begin{equation}
		\mathrm{length}_{\omega_{\can}}(\gamma)\leq d_{\can}(x,y)+\delta.
	\end{equation}
\end{lemma}
\begin{proof}
		Given $\delta>0$, by the definition of metric completion, we know that there exists $\delta'>0$ such that given points $x, y\in Y\setminus V_{\epsilon/2}$, there exists $\epsilon'=\epsilon'(x,y)>0$ such that for any points $x'\in B^{d_{\can}}(x,\delta')$ and $y'\in B^{d_{\can}}(y,\delta')$, there exists a smooth curve $\gamma_{x',y'}$ contained in $Y\setminus V_{\epsilon'}$, connecting $x'$ and $y'$ such that
		\begin{equation}
			\mathrm{length}(\gamma_{x',y'})\leq d_{\can}(x',y')+\delta
		\end{equation}Since $\overline{Y\setminus V_{\epsilon}}\times \overline{Y\setminus V_{\epsilon}}$ is compact, we can cover it by finitely many balls $B^{d_{\can}}(x_i,\delta')\times B^{d_{\can}}(y_i,\delta')$ for $i=1,\cdots, N$. We then define
\begin{equation}
    \epsilon' := \min \{ \epsilon(x_i,y_i) \mid i = 1,\dots,N \}.
\end{equation}
It is clear that this choice satisfies the desired property.\end{proof}

Next, we show that the Theorem \ref{main} can be reduced to proving an upper bound of the $d_{\can}$-distance.

\begin{lemma}\label{lem--dcan bound by dt}
	Suppose the following is true: given any small $\ve>0$ and two points $p, q\in X\setminus \tilde V_{\epsilon}$, we have
	\begin{equation}\label{eq--to prove}
		d_{\can}(f(p),f(q))\leq d_t(p,q)+\Psi(t^{-1}|\epsilon).
\end{equation}
In this case, we then have
\begin{equation}\label{agogna}
	\lim_{t\rightarrow\infty}d_{\gh}((X,d_t),(Y,d_{\can}))=0
\end{equation}
\end{lemma}
\begin{proof}
	Thanks to Lemma \ref{lem--gh close for smooth exhaustion}, to prove \eqref{agogna} it is enough to show that
	\begin{equation}\label{eq--intermediate}
		d_{\gh}((X\setminus \tilde V_{\epsilon},d_t),  (Y\setminus V_{\epsilon},d_{\can}))=\Psi(t^{-1}| \epsilon).
	\end{equation} To show this, we will show that the fibration map
	\begin{equation}
		f:X\setminus \tilde V_{\epsilon}\rightarrow Y\setminus V_{\epsilon}
	\end{equation}is the Gromov-Hausdorff approximation that we want.
Indeed, by Lemma \ref{lem--convexity on y}, and the $C^0_{\loc}$-convergence of $\omega(t)$ to $\omega_{\can}$ on $f^{-1}(Y^{\circ})$ in \eqref{conv}, we know that for any $p,q\in X\setminus \tilde V_{\epsilon}$, we have
	\begin{equation}
		d_t(p,q)\leq d_{\can}(f(p),f(q))+\Psi(t^{-1}|\epsilon).
	\end{equation}
Combining this with \eqref{eq--to prove} then gives
	\begin{equation}
		\left|d_t(p,q)-d_{\can}(f(p),f(q))\right|=\Psi(t^{-1}|\epsilon),
	\end{equation}
and since clearly $f$ is surjective, we obtain the desired Gromov-Hausdorff approximation, which shows \eqref{eq--intermediate}.
\end{proof}

\subsection{Further reduction to an upper bound of $d_{\can}$ by the $\mathcal L$-distance}\label{sec--reparametrize}

In the previous subsection, Theorem \ref{main} has been reduced to proving \eqref{eq--to prove}. In this section, which still follows closely \cite{LT}, we show that \eqref{eq--to prove} would follow from the ``key estimate'' in \eqref{key}, which is a sharp upper bound for $d_{\can}$ in terms of Perelman's $\mathcal L$-distance. The key estimate will be proved in the next section.

As in \cite{LT}, we reparametrize the flow in a standard way. Given $T\gg 1$, let $g(t)$ be the Riemannian metric defined by $\omega(t)$, and define
\begin{equation}\label{eq-reparametrization}
	\tilde{g}(s):=e^{t-T} g(t), \quad s:=\frac{1}{2}\left(e^{t-T}-1\right),
\end{equation}
which solve the unnormalized Ricci flow
\begin{equation}\label{eq--standard Ricci-flow}
	\frac{\partial}{\partial s} \tilde{g}(s)=-2 \operatorname{Ric}(\tilde{g}(s)), \quad s \geqslant s_T:=\frac{1}{2}\left(e^{-T}-1\right),
\end{equation}
with $\tilde{g}(0)=g(T)$. We clearly have
\begin{equation}
g(t)=\frac{\tilde{g}(s)}{1+2 s}, \quad t=T+\log (1+2 s),
\end{equation}
and \eqref{stscal} becomes
\begin{equation}\label{stscal2}
\sup_X|R(\tilde{g}(s))| \leqslant \frac{C}{1+2 s},\quad s\geq s_T.
\end{equation}
We let $\tau=-s$, so that the metrics $\tilde{g}(\tau)$ solve the backwards Ricci flow
\begin{equation}
\frac{\partial}{\partial \tau} \tilde{g}=2 \operatorname{Ric}(\tilde{g}),\left.\quad \tilde{g}\right|_{\tau=0}=g(T).
\end{equation}
Following Perelman, one defines the $\mathcal L$-length of a curve $\gamma(\tau)$ in $X$ (which, following Perelman, we think of as the curve $(\gamma(\tau),\tau)$ in space-time)  by
\begin{equation}
\mathcal{L}(\gamma)=\int \sqrt{\tau} (R(\ti{g}(\tau))+ |\partial_\tau \gamma|^2_{\ti{g}(\tau)}) d\tau,
\end{equation}
 and the $\mathcal L$-distance for two points in space-time as the infimum of such. We will use the following convention.
 Fix two parameters
\begin{equation}
T \gg 1 \gg \ov{\tau} > 0.
\end{equation}
By an $\mathcal L$-geodesic $\gamma$ in space-time from $(p,0)_T$ to $(q,\ov{\tau})_T$, we mean that, after performing the reparametrization \eqref{eq-reparametrization}, the curve $\gamma$ is an $\mathcal L$-geodesic from $(p,0)$ to $(q,-\ov{\tau})$ in the space-time of the standard Ricci flow \eqref{eq--standard Ricci-flow}. We will denote by $L_T(q,\ov{\tau})$ the $\mathcal{L}$-distance between $(p,0)_T$ and $(q,\tau)_T$.

The following is then the key estimate that we shall prove:

\begin{proposition}[Key estimate]\label{prop--main to prove}
	For any two points $p, q\in X\setminus \tilde V_{\epsilon}$, we have
	\begin{equation}\label{key}
		L_T(q,\ov{\tau})\geq \frac{1}{2\sqrt{\ov{\tau}}}d_{\can}(f(p),f(q))^2+\Psi(T^{-1}|\epsilon,\ov{\tau})+\Psi(\bar\tau|\epsilon).
	\end{equation}
\end{proposition}

Following the argument in \cite{LT}, we now show that Proposition \ref{prop--main to prove} implies Theorem \ref{main}:
\begin{proof}[Proof of Theorem \ref{main}, assuming Proposition \ref{prop--main to prove}]
Thanks to Lemma \ref{lem--dcan bound by dt}, it suffices to show that \eqref{eq--to prove} holds. For this, we first observe the following easy upper bound of the $\mathcal L$-distance. By Lemma \ref{lem--convexity on y}, given $p,q\in  X\setminus \tilde V_{\epsilon}$, we can find a curve $\gamma$ contained in $Y\setminus V_{\epsilon'}$ connecting $f(p)$ and $f(q)$ such that
\begin{equation}
	\mathrm{length}_{\omega_{\can}}(\gamma)\leq d_{\can}(f(p),f(q))+\Psi(\epsilon'|\epsilon).
\end{equation}
Then using the fiber bundle structure of $f$ and the $C^0_{\loc}$ convergence of $\omega(t)$, arguing as in \cite[Proof of Lemma 9.1]{TWY2}, we can construct a curve $\tilde \gamma$ contained in $X^\circ$, connecting $p$ and $q$, such that
\begin{equation}\label{eq--length upper bound}
	\mathrm{length}_{\omega(T)}(\tilde \gamma)\leq d_{\can}(f(p),f(q))+\Psi(\epsilon'|\epsilon)+\Psi(T^{-1}|\epsilon').
\end{equation}Note that this curve $\tilde\gamma$ depends on the parameter $\epsilon'$ and for simplicity of notation, in the following, we omit the dependence on $\epsilon'$.

We parametrize $\tilde \gamma$ by $\tau \in [0,\ov{\tau}]$ such that $|\de_\tau\ti{\gamma}|_{\ti{g}(\tau)}=\frac{A}{2\sqrt{\tau\ov{\tau}}}$ for all $0<\tau\leq\ov{\tau}$, where \begin{equation}
A=\int_0^{\ov{\tau}}|\de_\tau\ti{\gamma}|_{\ti{g}(\tau)}d\tau.
\end{equation}
By the definition of $\tilde g(\tau)$ and the asymptotics \eqref{twyconv}, we know that for $\tau\in [0,\bar\tau]$,
\begin{equation}\label{eq--comparable}
\begin{aligned}
	\tilde g(\tau)&=(1-2\tau) g(T+\log(1-2\tau))\\&\leq (1+\Psi(T^{-1}|\epsilon))(f^*\omega_{\can}+e^{-T}(1-2\tau)^{-1}\omega_{\rm SRF})\\
	&\leq (1+4\ov{\tau})(1+\Psi(T^{-1}|\epsilon)) g(T)\\
	&=(1+4\ov{\tau})(1+\Psi(T^{-1}|\epsilon))\tilde g(0)
\end{aligned}
\end{equation}Combining this with \eqref{eq--length upper bound}, we see that
\begin{equation}
	A\leq (1+4\ov{\tau}+\Psi(T^{-1}|\epsilon))d_{\rm can}(f(p),f(q))+\Psi(T^{-1}|\ov{\tau},\epsilon).
\end{equation}
Using the scalar curvature bound \eqref{stscal2}, we can then estimate
\begin{equation}\label{eq--upper bound of ldistance}
	\begin{aligned}
		L_T(q,\ov{\tau}) &\leq\mathcal{L}(\ti{\gamma})
\leq C\ov{\tau}^{\frac{3}{2}}+\int_0^{\ov{\tau}}\sqrt{\tau}|\de_\tau\ti{\gamma}|^2_{\ti{g}(\tau)}d\tau\leq C\ov{\tau}^{\frac{3}{2}}+\frac{A^2}{2\sqrt{\ov{\tau}}}\\
&\leq \frac{1}{2\sqrt{\ov{\tau}}}d_{\rm can}(f(p),f(q))^2+\Psi(T^{-1}|\ov{\tau},\epsilon)+\Psi(\ov{\tau}|\epsilon).
	\end{aligned}
\end{equation}
Now, using the key estimate \eqref{key} and the triangle inequality, we see that given any small $\delta>0$ and any point $q'\in X$ with $f(q')\in B^{d_{\can}}(f(q),\delta)$, we have
\begin{equation}
		L_T(q',\ov{\tau})\geq \frac{1}{2\sqrt{\ov{\tau}}}\left(d_{\can}(f(p),f(q))^2-\Psi(\delta|\epsilon)\right)+\Psi(T^{-1}|\epsilon,\ov{\tau})+\Psi(\bar\tau|\epsilon).
\end{equation}
	We consider
	\begin{equation}
		\ov{L}(q',\tau):=2\sqrt{\tau}L_T(q',\tau),
	\end{equation}which satisfies
\begin{equation}\label{eq--heat equation for barL}
\left(\frac{\partial}{\partial \tau}+\Delta_{\ti{g}(\tau)}\right)\ov{L} \leq 4n,
\end{equation}
\begin{equation}\label{lower bound for lbart}
	\ov{L}(q',\ov{\tau})\geq d_{\can}(f(p),f(q))^2+\Psi(T^{-1}|\epsilon,\ov{\tau})+\Psi(\bar\tau|\epsilon)+\Psi(\delta|\epsilon),
\end{equation}and
\begin{equation}\label{gimme}
	\lim_{\tau\rightarrow 0^+}\ov{L}(q',\tau)=d_{\tilde g(0)}(p,q')^2=d_T(p,q')^2.
\end{equation}

Let $\chi$ be a smooth (time-independent) cutoff function on $Y$ supported in $B^{d_{\can}}(f(q),2\delta)$ and equal to $1$ on $B^{d_{\can}}(f(q),\delta)$, and denote by the same symbol its pullback to $X$ via $f$. We know that
	$-C(\epsilon) \omega_Y\leq \delta^2\ii \partial\db \chi\leq C(\epsilon) \omega_Y,$
hence by \eqref{schwarz} we have
\begin{equation}\label{eq--bound on cut-off}
\sup_X|\Delta_{\ti{g}(\tau)} \chi|\leq C(\epsilon)\delta^{-2}.
\end{equation}Using \eqref{eq--heat equation for barL}, we obtain
\begin{equation}\begin{split}
\int_X \chi \ov{L}(\cdot,0)\ti{\omega}^n(0)
  \geq&\int_X \chi \ov{L}(\cdot,\ov{\tau})\ti{\omega}^n(\ov{\tau}) + \int_0^{\ov{\tau}} \int_X  \ov{L}(\cdot,\tau) \Delta_{\ti{g}(\tau)} \chi\ti{\omega}^n(\tau) d\tau\\
  &-4n \int_0^{\ov{\tau}} \int_X  \ov{L}(\cdot,\tau) \chi\ti{\omega}^n(\tau) d\tau \\
  &-2\int_0^{\ov{\tau}}\int_X \chi\ov{L}(\cdot,\tau) R(\ti{g}(\tau))\ti{\omega}^n(\tau) d\tau,
\end{split}
\end{equation}
Then using the upper bound on $\ov{L}$ in \eqref{eq--upper bound of ldistance}, the scalar curvature bound \eqref{stscal2}, \eqref{lower bound for lbart} and \eqref{eq--bound on cut-off}, arguing as in \cite{LT} we can obtain that
\begin{equation}\label{integral bound}
\begin{aligned}
	\int_X \chi \ov{L}(\cdot,0)\ti{\omega}^n(0)\geq  &\left(d_{\rm can}(f(p),f(q))^2+\Psi(T^{-1}|\epsilon,\ov{\tau})+\Psi(\ov{\tau}|\delta,\epsilon) \right)\int_X \chi \ti{\omega}^n(\ov{\tau}).
\end{aligned}
\end{equation}
Note that for $q'\in X$ with $f(q')\in B^{d_{\can}}(f(q),\delta)$, we have
\begin{equation}\label{pointwise lower bound of lbar}
	d_{T}(p,q)= d_{T}(p,q')+\Psi(T^{-1},\delta|\epsilon),
\end{equation}then using \eqref{gimme}, we obtain
\begin{equation}\label{due}
\int_X \chi \ov{L}(\cdot,0)\ti{\omega}^n(0)=\int_X \chi d_T(p,\cdot)^2\ti{\omega}^n(0)\leq (d_T(p,q)+\Psi(T^{-1}|\delta,\ve))^2\int_X\chi \ti{\omega}^n(0),
\end{equation}
and using the scalar curvature lower bound, we have
\begin{equation}\label{difference of volume}
	\int_X \chi \ti{\omega}^n(\ov{\tau})\geq (1-\Psi(\ov{\tau}|\delta,\epsilon))\int_X\chi \ti{\omega}^n(0).
\end{equation}Combining \eqref{integral bound}, \eqref{due}, \eqref{pointwise lower bound of lbar}, and \eqref{difference of volume}, and choosing $\delta$ sufficiently small, then $\ov{\tau}$ sufficiently small, and finally $T$ sufficiently large, we obtain the desired estimate in \eqref{eq--to prove}:
\begin{equation}
    d_T(p,q) \ge d_{\can}(f(p),f(q)) + \Psi(T^{-1}| \epsilon).
\end{equation}
\end{proof}

\section{Proof of the main theorem}\label{sec--main}
In this section we give the proof of Proposition \ref{prop--main to prove}, which implies Theorem \ref{main} by the discussion in  Section \ref{sec--reduction}.
In the following, we always consider points $p,q \in X \setminus \tilde V_{\epsilon}$ and parameters $\delta$ satisfying
\begin{equation}
    \delta \ll \epsilon.
\end{equation}
We note that by \eqref{useless}, for the given parameters $\epsilon$ and $\delta$, and for all sufficiently large $T$, we have
\begin{equation}
    f^{-1}\big(B^{d_{\can}}(f(q),\delta/2)\big)
    \subset B^{d_T}(q,\delta)
    \subset f^{-1}\big(B^{d_{\can}}(f(q),2\delta)\big).
\end{equation}
Therefore, in the following we will not distinguish between
$f^{-1}\big(B^{d_{\can}}(f(q),\delta)\big)$ and $B^{d_T}(q,\delta)$.
For simplicity of notation, we will also omit the dependence on $\epsilon$ in Proposition \ref{prop--main to prove} throughout this section.
\subsection{An almost-avoidance principle} In this subsection, we show that for a family of sets $\{U_\eta\}_{\eta>0}$ satisfying certain properties, namely small volume and a separation between $\partial U_\eta$ and $\partial U_{\eta/2}$, a typical $\mathcal L$-geodesic connecting $p$ and $q$ intersects $f^{-1}(U_\eta \setminus U_{\eta/2})$ at most $\eta^{-\varepsilon}$ times. Since we will apply this result three times to different families of sets, we find it convenient to abstract the properties that we use, as follows.

Throughout this section, we will consider $\{U_\eta\}_{\eta\in (0,\eta_0]}$ a family of open subsets of $Y$
satisfying the following conditions
\begin{enumerate}
\item Nestedness: $U_{\eta'}\subset U_\eta$  for $\eta'\leq \eta$.
	\item Minkowski content bound: there exists $\rho\in (0,1)$ and $C>0$, such that \begin{equation}
		\vol(U_\eta, \omega_{\can}^{m})\leq C\eta^{2-\rho},
	\end{equation}
for all $0<\eta\leq\eta_0$.
	\item Separation: there exists $c>0$ such that
	\begin{equation}
		 d_{\can}(x, \partial  U_{\eta})\geq c\eta,
	\end{equation}
 for any $0<\eta\leq\eta_0$ and $x\in U_{\eta/2}$.
	\item Normalization: $\left(B^{d_{\can}}(f(p),\delta)\cup B^{d_{\can}}(f(q),\delta)\right)\cap U_{\eta_0}=\emptyset$.
\end{enumerate}
We will also consider the following stronger condition:
\begin{itemize}
\item[(2')] There exists $\rho\in (0,1)$ and $C>0$, such that \begin{equation}
		\vol(U_\eta, \omega_{\can}^{m})\leq C\eta^{3-\rho},
	\end{equation}
for all $0<\eta\leq\eta_0$.
\end{itemize}

We set
\begin{equation}
	\tilde U_{\eta}=f^{-1}(U_\eta)\subset X.
\end{equation}
\begin{defn}
	For a piecewise differentiable curve $\gamma$ inside $X$, we say $\gamma$ has an $\eta$-event with respect to $\{U_{\eta}\}$ if $f(\gamma)$ enters $U_{\eta}$ and reaches $U_{\frac{\eta}{2}}$ before returning to the boundary of $U_{\eta}$.
\end{defn}
	
As explained in Section \ref{sec--reparametrize}, we rescale and reparametrize the flow to obtain an unnormalized Ricci flow and study its $\mathcal{L}$-geodesics.
We use the notation introduced there throughout the following discussion. The following Lemma is analogous to \cite[(4.31)]{LT}:

\begin{lemma}\label{lem--lower bound of first bad time}
Let $\gamma$ be a minimizing $\mathcal L$-geodesic from $(p,0)_T$ to $(q',\ov{\tau})_T$,
where $q'$ satisfies $f(q') \in B^{d_{\can}}(f(q),\delta)$.
Let $\ov{\tau}'$ denote the first time at which $f(\gamma)$ exits the ball
$B^{d_{\can}}(f(p),\delta)$.
Then, for $\ov{\tau}$ sufficiently small and $T$ sufficiently large, there exists a constant $C = C(\delta)$ such that
\begin{equation}
    \ov{\tau}' \ge C^{-1}\ov{\tau}.
\end{equation}
\end{lemma}

\begin{proof}Using \eqref{twyconv}, \eqref{stscal2} and the upper bound of $\mathcal L$-geodesics \eqref{eq--upper bound of ldistance}, we have
	\begin{equation}\begin{split}
\delta &\leq (1+\Psi(T^{-1},\ov{\tau}))\int_0^{\ov{\tau}'}|\de_\tau\gamma|_{\ti{g}(\tau)}d\tau\\
&\leq C\left(\int_0^{\ov{\tau}'}\sqrt{\tau}|\de_\tau\gamma|^2_{\ti{g}(\tau)}d\tau\right)^{\frac{1}{2}}\left(\int_0^{\ov{\tau}'}\frac{1}{\sqrt{\tau}}d\tau\right)^{\frac{1}{2}}\\
&\leq C\ov{\tau}'^{\frac{1}{4}}\left(C\ov{\tau}'^{\frac{3}{2}}+\int_0^{\ov{\tau}'}\sqrt{\tau}(R(\ti{g}(\tau))+|\de_\tau\gamma|^2_{\ti{g}(\tau)})d\tau\right)^{\frac{1}{2}}\\
&\leq C\ov{\tau}'^{\frac{1}{4}}\left(C\ov{\tau}'^{\frac{3}{2}}+C\ov{\tau}^{-\frac{1}{2}}\right)^{\frac{1}{2}}\\
&\leq C\ov{\tau}'^{\frac{1}{4}}\ov{\tau}^{-\frac{1}{4}}.
\end{split}\end{equation}
\end{proof}

In the following, $\ov{\tau}$ will always be chosen sufficiently small (depending on $\delta$) so that Lemma \ref{lem--lower bound of first bad time} holds.

The following observation will be crucial for our arguments:
\begin{proposition}\label{prop-with controlled events}
	Let $W \subset f^{-1}(B^{d_{\can}}(f(q),2\delta))$ be a Borel subset with positive volume, i.e.
	\begin{equation}\label{positive volume}
		\vol(W,\omega_0^n)>0.
	\end{equation} and let $\{U_\eta\}$ be a family of subsets of $Y$ satisfying (1)-(4) above and fix $\epsilon>\rho/2$, where $\rho$ is as in condition (2) above. Then for
	\begin{equation}
		T^{-1}\ll \eta\ll \ov{\tau}\ll \delta,
	\end{equation} there exists a locally closed subset  $\Omega=\Omega(T,\eta,\bar\tau,\delta)\subset W$ with volume
\begin{equation}
	\vol(\Omega,\omega(T)^n)\geq \frac{1}{2}\vol(W,\omega(T)^n)
\end{equation}such that for any $q'\in \Omega$, there exists a minimizing
 $\mathcal L$-geodesic from $(p,0)_T$ to $(q', \ov{\tau})_T$ for which the number of $\eta$-events with respect to $\{U_\eta\}$ is at most $\eta^{-\epsilon}$.

 If, on the other hand, we assume that the family $\{U_\eta\}$ satisfies (1),(2'),(3),(4), then for $T^{-1}\ll \eta\ll \ov{\tau}\ll \delta$ there is $\Omega\subset W$ as above such that for any $q'\in \Omega$, there exists a minimizing
 $\mathcal L$-geodesic from $(p,0)_T$ to $(q', \ov{\tau})_T$ which is disjoint from $\ti{U}_{\eta/2}$.
\end{proposition}
\begin{proof}
First, we assume properties (1)-(4).	 Let $\Omega_{\ov{\tau}}\subset T_pX$ be the open set such that the $\mathcal L$-exponential map $\mathcal L \exp_{p,\ov{\tau}}$ restricted to $\Omega_{\ov{\tau}}$ gives a diffeomorphism onto $W^{\rm reg}$, which is an open subset of $W$ such that  $W\setminus W^{\rm reg}$ has measure zero. Moreover we know that for a point $q'\in W^{\reg}$, there exists a unique $\mathcal L$-geodesic from $(p,0)_T$ to $(q',\ov{\tau})_T$. Let $\Omega_{\ov{\tau}}' \subset \Omega_{\ov{\tau}}$ denote the subset consisting of those initial tangent vectors for which the associated $\mathcal L$-geodesic has more than $\eta^{-\varepsilon}$ $\eta$-events with respect to the family $\{U_\eta\}$. By the continuous dependence of $\mathcal L$-geodesics on their initial data, $\Omega_{\ov{\tau}}'$ is an open subset of $\Omega_{\ov{\tau}}$.

	 For any $v\in \Omega_{\ov{\tau}}'$, there exists a union of open intervals $I_v\in [0,\ov{\tau}]$ such that $\mathcal L \exp_{p,\tau}(v)\in \tilde U_{\eta}$ exactly when $\tau\in I_v$. By Lemma \ref{lem--lower bound of first bad time}, for any $v\in \Omega_{\ov{\tau}}'$, we have
	\begin{equation}\label{control of interval}
		I_v\subset[C^{-1}\ov{\tau}, \ov{\tau}].
	\end{equation}We consider
	\begin{equation}
		\mathcal T=\{(v, \tau)\mid v\in \Omega_{\ov{\tau}}',  \tau \in I_v\}\subset T_pX\times [0,\ov{\tau}].
	\end{equation}Since $I_v$ depends continuously on $v$, the set $\mathcal T$ is a measurable set and in the following, we can use Fubini's theorem on $\mathcal T$.
We have a smooth injective map
	\begin{equation}
		\mathcal L\exp_p : \mathcal T\rightarrow \tilde U_\eta\times [C^{-1}\ov{\tau},\ov{\tau}],
	\end{equation}which maps $(v,\tau)$ to $(\mathcal L \exp_{p,\tau}(v),\tau)$. Recall Perelman's monotonicity \cite{perelman},
	\begin{equation}\label{jacob2}
J(v,\tau)\geq \left(\frac{\tau}{\ov{\tau}}\right)^me^{\ell(v,\tau)-\ell(v,\ov{\tau})}J(v,\ov{\tau}),
\end{equation}where $J(v,\tau)$ denotes the Jacobian of $\mathcal L\exp_{p,\tau}$ and
\begin{equation}
\ell(v,\tau):=\frac{1}{2\sqrt{\tau}}L_T(\mathcal{L}\exp_{p,\tau}(v),\tau),
\end{equation} is Perelman's reduced length. Then combining this
with the upper bound on the $\mathcal L$-distance \eqref{eq--upper bound of ldistance} and Lemma \ref{lem--lower bound of first bad time}, we obtain
\begin{equation}
	J(v,\tau)\geq C^{-1}e^{-\frac{C}{\ov{\tau}}}J(v,\ov{\tau}).
\end{equation}We compute
	\begin{equation}\label{eq--fubini}
		\begin{aligned}
		\int_{C^{-1}\ov{\tau}}^{\ov{\tau}}\vol(\tilde U_\eta,\tilde \omega(\tau)^n)d\tau &\geq \int_{\Omega_{\ov{\tau}}'}\int_{I_v}J(v,\tau)d\tau dv\\
		&\geq C^{-1}e^{-\frac{C}{\ov{\tau}}}\int_{\Omega_{\ov{\tau}}'}\int_{I_v}J(v,\ov{\tau})d\tau dv\\
		&\geq C^{-1}e^{-\frac{C}{\ov{\tau}}} \min_{v\in \Omega_{\ov{\tau}}'}|I_v| \vol(\mathcal L\exp_{p,\ov{\tau}}(\Omega_{\ov{\tau}}'),\tilde\omega(\ov{\tau})^n).
		\end{aligned}	\end{equation}
		We want to have a lower bound for $|I_v|$. Since we have at least $\eta^{-\ve}$ $\eta$-events and the $d_{\can}$-length of a curve for each $\eta$-event is at least $\eta/2$ (by property (3)), then by \eqref{eq--comparable}, we obtain that for $T$ sufficiently large (depending on $\eta$) and $\bar\tau\in (0,1/10]$, we get the following
		\begin{equation}\label{korko}
			\eta^{1-\epsilon}\leq 4\int_{I_v}|\partial_\tau \gamma|_{\tilde g(\tau)}d\tau.
		\end{equation}
		Using Cauchy-Schwarz and the  upper bound of the $\mathcal L$-distance \eqref{eq--upper bound of ldistance}, we get
		\begin{equation}
			\begin{aligned}
\int_{I_v}\left|\partial_\tau \gamma\right|_{\tilde{g}(\tau)} d \tau & \leqslant  C\ov{\tau}^{-\frac{1}{4}} \int_{I_v} \tau^{\frac{1}{4}}\left|\partial_\tau \gamma\right|_{\tilde{g}(\tau)} d \tau \\
& \leqslant  C\ov{\tau}^{-\frac{1}{4}}\left(\int_{I_v} \sqrt{\tau}\left|\partial_\tau \gamma\right|_{\tilde{g}(\tau)}^2 d \tau\right)^{\frac{1}{2}}|I_v|^{\frac{1}{2}}\\
&\leq C\bar\tau^{-1/2}|I_v|^{\frac12},
\end{aligned}
		\end{equation} and together with \eqref{korko} we see that
		\begin{equation}\label{eq--lower bound of intervals}
			|I_v|\geq C^{-1}\ov{\tau}\eta^{2(1-\epsilon)}.
		\end{equation}
Plugging this back into \eqref{eq--fubini}, we get
\begin{equation}\label{fut2}\begin{aligned}
\vol(\mathcal L\exp_{p,\ov{\tau}}(\Omega_{\ov{\tau}}'),\tilde\omega(\ov{\tau})^n)&\leq Ce^{\frac{C}{\ov{\tau}}}\ov{\tau}^{-1}\eta^{-2(1-\ve)}\int_{C^{-1}\ov{\tau}}^{\ov{\tau}}\vol(\tilde U_\eta,\tilde \omega(\tau)^n)d\tau\\
&\leq Ce^{\frac{C}{\ov{\tau}}}\eta^{-2(1-\ve)}\vol(\tilde U_\eta, \omega(T)^n).
\end{aligned}
\end{equation}
But using \eqref{fut}, \eqref{volscal} and assumption (2) we can bound
\begin{equation}\label{ut}\begin{aligned}
\vol(\tilde U_\eta, \omega(T)^n)&\leq C\vol(X, \omega(T)^n)\vol(U_\eta,\omega_{\rm can}^m)\\
&\leq Ce^{-(n-m)T}\eta^{2-\rho},
\end{aligned}
\end{equation}
and inserting this into \eqref{fut2} we obtain that
\begin{equation}\label{fut3}
	\vol(\mathcal L\exp_{p,\ov{\tau}}(\Omega_{\ov{\tau}}'),\tilde\omega(\ov{\tau})^n)<\Psi(\eta|\bar\tau,\delta) e^{-(n-m)T}.
\end{equation}By the estimate on the volume form \eqref{volscal} and the initial assumption on $W$ \eqref{positive volume}, we know that for some $c>0$ independent of $T$,
\begin{equation}\vol(W,\omega(T)^n)\geq ce^{-(n-m)T}.
\end{equation}
Therefore, if we choose $\eta$ sufficiently small, we have thus proved that there exists a closed subset  $\Omega\subset W$ with volume
\begin{equation}
	\vol(\Omega,\omega(T)^n)\geq \frac{1}{2}\vol(W, \omega(T)^n),
\end{equation}such that for any $q'\in \Omega$, there exists a minimizing
 $\mathcal L$-geodesic from $(p,0)_T$ to $(q', \ov{\tau})_T$ which has at most $\eta^{-\epsilon}$ $\eta$-events with respect to $\{U_\eta\}$, as desired.

Now, we assume (1),(2'),(3),(4), and we run the same argument we just did, taking now $\Omega_{\ov{\tau}}' \subset \Omega_{\ov{\tau}}$ denote the subset consisting of those initial tangent vectors for which the associated $\mathcal L$-geodesic has at least one $\eta$-event with respect to the family $\{U_\eta\}$. Then \eqref{fut2} above changes to
\begin{equation}
\vol(\mathcal L\exp_{p,\ov{\tau}}(\Omega_{\ov{\tau}}'),\tilde\omega(\ov{\tau})^n)\leq Ce^{\frac{C}{\ov{\tau}}}\eta^{-2}\vol(\tilde U_\eta, \omega(T)^n),
\end{equation}
while thanks to (2') the estimate \eqref{ut} changes to
\begin{equation}
\vol(\tilde U_\eta, \omega(T)^n)\leq Ce^{-(n-m)T}\eta^{3-\rho},
\end{equation}
for some $0<\rho<1$, which again gives the same estimate as in \eqref{fut3}. We thus get $\Omega\subset W$ as above such that for any $q'\in \Omega$, there exists a minimizing
 $\mathcal L$-geodesic from $(p,0)_T$ to $(q', \ov{\tau})_T$ which has no $\eta$-events with respect to $\{U_\eta\}$, i.e. it is disjoint from $\ti{U}_{\eta/2}$, as desired.
\end{proof}

\subsection{Integrability of $\omega_{\can}$ on the almost regular set}
Recall that $Y^{\rm reg}$ denotes the points where the variety $Y$ is smooth, and by definition we have $Y^{\circ}\subset Y^{\reg}$.
Following \cite{LiuSz},  we define the almost regular set as follows: given $\epsilon>0$, let
\begin{equation}\label{reg}
	\mathcal R_\epsilon:=\left\{y\in Y\bigg| \lim_{r\rightarrow 0}\frac{\vol(B^{d_{\can}}(y,r),\omega_{\rm can}^m)}{\omega_{2m}r^{2m}}>1-\epsilon\right\}.
\end{equation}
Clearly $\mathcal R_\epsilon$ is open and for $\epsilon'<\epsilon$, we have
\begin{equation}
	\mathcal R_{\epsilon'}\subset \mathcal R_{\epsilon}.
\end{equation}

The following result is proved in \cite{Sz,LiuSz}, building on the earlier work in \cite{DS1,CDS2}.  Note that as discussed in \cite[Section 2]{Sz}, by using the argument in \cite{CDS2,LiuSz}, we know that for any iterated tangent cone $V$ of $Y$, the singular set $V\setminus \mathcal R_{\epsilon}(V)$ has capacity zero. Therefore the results in \cite[Section 4]{LiuSz} can be applied to $(Y, d_{\can})$.

\begin{theorem}[\cite{LiuSz,Sz}]\label{thm--local holder}
	There exists $\epsilon_1>0$ such that
	\begin{equation}
		\mathcal R_{\epsilon_1} \subset Y^{\reg}.
	\end{equation}Moreover for each $y\in Y^{\reg}$, there exists $C=C(Y,y)$ and $\alpha=\alpha(Y,y)>0$ and holomorphic coordinates $z=(z_1\cdots,z_m)$ around $y$ such that
	\begin{equation}\label{eq--distance comparable}
		C^{-1}|z|\leq d_{\can}(y,\cdot)\leq C |z|^{\alpha}.
	\end{equation}
\end{theorem}

The main result in this section is the following.
\begin{lemma}\label{lem--high integrability}
	For any $p>0$, there exists $\epsilon=\epsilon(p)$ such that
	\begin{equation}
		\frac{\omega_{\can}^m}{\omega_Y^m}\in L_{\loc}^p(\mathcal R_{\epsilon},\omega_Y^m).
	\end{equation}
\end{lemma}

\begin{proof}
	This basically is contained in \cite[P.268--270]{LiuSz}. We follow their argument closely.	
	
For any small $\epsilon>0$, there exists $\epsilon'\ll \epsilon$ such that for any $y\in \mathcal R_{\epsilon'}$ the ball $B\left(y, \frac{1}{\epsilon}\right)$ (with the distance $d_{\rm can}$) is $\epsilon$-Gromov-Hausdorff close to a ball in $\mathbb{C}^m$,  by scaling. Then we can find a holomorphic map $F=\left(f_1, \ldots, f_m\right)$ on $B(y, 100)$ which gives a $\Psi(\epsilon )$-Gromov Hausdorff approximation onto its image in $\mathbb{C}^m$ and defines a holomorphic chart on $B(y, 1)$. Moreover, we may assume $f_j(y)=0, \int_{B(p, 1)} f_j \overline{f_k}=0$ for $j \neq k$. We have

\begin{equation}
\frac{f_{B(y, 2)}\left|f_j\right|^2}{f_{B(y, 1)}\left|f_j\right|^2} \leq 4+\Psi(\epsilon).
\end{equation}

Using a three annulus type argument and the fact that the tangent cone at $y$ is close to $\mathbb C^m$ (see \cite[Lemma 2.4]{LiuSz} and \cite[Proposition 3.7]{DS2}), we see that for any $0<r<1$, we have
\begin{equation}\label{eq--lower vanishing order}
	4-\Psi(\epsilon) \leq \frac{f_{B(y, 2 r)}\left|f_j\right|^2}{f_{B(y, r)}\left|f_j\right|^2} \leq 4+\Psi(\epsilon).
\end{equation}
For any given $r>0$, we can choose the functions $f_1, \ldots, f_m$ so that they are orthogonal simultaneously with respect to the $L^2$ inner products on $B(y, 1)$ and $B(y, r)$. Now let $c_j=c_j(r)$ be the constants so that $\sup _{B(y, r)}\left|c_j f_j\right|=r$. By \eqref{eq--lower vanishing order}, we know that
\begin{equation}
	|c_j(r)|\leq Cr^{-\Psi(\epsilon)}.
\end{equation}
Define $f_j^{\prime}=c_j f_j$. Then as in \cite{LiuSz} we see that $F^{\prime}=\left(f_1^{\prime}, \cdots f_m^{\prime}\right)$ is a $\Psi(\epsilon) r$-Gromov-Hausdorff approximation to a ball in $\mathbb{C}^m$ and an estimate of Cheeger-Colding implies that
\begin{equation}
\fint_{B(y, r)} ||d f_1^{\prime} \wedge d f_2^{\prime} \wedge \ldots \wedge d f_m^{\prime}|_{g_{\rm can}}-1|=\Psi(\epsilon),
\end{equation}
and in particular, we have
\begin{equation}
	\sup _{B(y, r)}\left|d f_1^{\prime} \wedge d f_2^{\prime} \wedge \ldots \wedge d f_m^{\prime}\right|_{g_{\rm can}} \geq c(n)>0.
\end{equation}
Therefore we have
\begin{equation}\label{eq--lower bound}
	\sup_{ B(y,r)} \left|d f_1 \wedge d f_2 \wedge \ldots \wedge d f_m\right|_{g_{\rm can}} \geq c(n) r^{\Psi(\epsilon)}.
\end{equation}
Since near any point of $Y$ we can write locally $\omega_{\can}=\ddbar\vp$ for some continuous potential $\varphi$, and $\Ric(\omega_{\can})\geq -\omega_{\can}$ on $Y^\circ$, we have that
\begin{equation}\label{22}
	\log |d f_1 \wedge d f_2 \wedge \ldots \wedge d f_m|_{g_{\rm can}}+\varphi
\end{equation}is a psh function on $Y^\circ$ near $y$. Furthermore, as a consequence of the Schwarz Lemma estimate \eqref{schwarz} together with the convergence in \eqref{conv} we have that $\omega_{\rm can}\geq C^{-1}\omega_Y$ on $Y^\circ$, hence the function $\frac{\omega_Y^m}{\omega_{\can}^m}$ is bounded above on $Y^\circ.$
This, together with the simple observation that
\begin{equation}\label{33}
\log |d f_1 \wedge d f_2 \wedge \ldots \wedge d f_m|_{g_{\rm can}}+\varphi\leq \log \frac{\omega_Y^m}{\omega_{\can}^m}+C,
\end{equation}
shows that the psh function in \eqref{22} is locally bounded above near $D\cap Y^{\rm reg}$, and so it extends to a psh function near $y$ on $Y^{\rm reg}$, which we denote with the same name. By \eqref{eq--lower bound} and the local distance estimate \eqref{eq--distance comparable}, we know that the Lelong number of this psh function at $y$ is bounded above by $\Psi(\epsilon)$, i.e.,
\begin{equation}\label{eq--lelong number upper bound}
	\nu(\log |d f_1 \wedge d f_2 \wedge \ldots \wedge d f_m|_{g_{\rm can}}+\varphi, y)\leq \Psi(\epsilon).
\end{equation}
Then for any given $p$, we can choose $\epsilon$ sufficiently small such that $p\Psi(\epsilon)\leq \frac{1}{2}$, where $\Psi(\epsilon)$ denotes the right-hand side of \eqref{eq--lelong number upper bound}. This then determines $\epsilon'$ and by the above argument, we know that for $y\in \mathcal R_{\epsilon'}$, we have
\begin{equation}
	\nu\left(\log |d f_1 \wedge d f_2 \wedge \ldots \wedge d f_m|_{g_{\rm can}}+\varphi,y\right)\leq \frac{1}{2p}.
\end{equation} Skoda's integrability theorem, together with \eqref{33}, then shows that
 $\frac{\omega_{\can}^m}{\omega_Y^m}$ is in $L^p(\omega_Y^m)$ in a neighborhood of $y$.
	\end{proof}

\subsection{The choice of the regions}
First, we recall that as a well-known and simple consequence of the Schwarz Lemma estimate \eqref{schwarz} together with the convergence in \eqref{conv}, we have
\begin{equation}\label{eq--holder estimate}
	d_Y\leq Cd_{\can},
\end{equation}
on the whole of $Y$.

Since \(\omega_{\can}\) satisfies a uniform lower bound on the volume-ratio, there exists
\(\gamma_0>0\) such that any metric cone of the form
\[
	\mathbb{R}^{2m-2}\times C(S^1_\gamma)
\]
arising as a rescaled limit of \((Y,\lambda_j\omega_{\can},p_j)\) for some $\lambda_j\rightarrow \infty$ and $p_j\in Y$, must satisfy
\(\gamma\geq \gamma_0\). A standard contradiction argument, together with
\cite[Proposition 11]{Sz} implies the following result, which is the RCD analogue of
\cite[Proposition 3.2]{LiuSz} for Ricci limit spaces.
\begin{proposition}\label{savior}
	There exists $\epsilon_0>0$ so that the following holds. Assume that for some $x\in Y$ and some $\ve\leq\ve_0$ we have
$$d_{\rm GH}\left(B\left(x, \epsilon^{-1}\right), B_V\left(o, \epsilon^{-1}\right)\right)<\epsilon,$$
for some metric cone $(V, o)$ that splits off an $\mathbb{R}^{2m-2}$-factor. Then $x\in Y^{\reg}$.
\end{proposition}

 As a consequence, we show that \(Y\setminus Y^{\reg}\) is contained in a quantitative singular set, consisting of  points whose centered balls are not close, at any scale, to metric cones splitting off an \(\mathbb{R}^{2m-2}\)-factor:
 \begin{lemma}\label{lem-singular set control}
 	There exists $0<\rho_0<1$ such that for all $0<r<1$, we have
 	\begin{equation}
 		Y\setminus Y^{\reg}\subset \mathcal S_{\rho_0,r}^{2m-3},
 	\end{equation}
where the latter is the effective singular stratum defined in \cite[Definition 1.2]{CN2013}.
 \end{lemma}
 \begin{proof}
 	This follows from a direct contradiction argument. Suppose not, then there would exist sequences $\rho_i\rightarrow 0$, $x_i\in Y\setminus Y^{\reg}$, $r_i\in (0,1)$ and $s_i\in (r_i,1)$ and a sequence of metric cones $(V_i, o_i)$ splitting off an $\mathbb R^{2m-2}$-factor such that
 	\begin{equation}
 		d_{\mathrm{GH}}(B(x_i,s_i), B(o_i, s_i))\leq \rho_i s_i.
 	\end{equation}Equivalently we would have
 	\begin{equation}
 		d_{\mathrm{GH}}(B(x_i,\rho_i^{-1/2}), B_{V_i}(o_i, \rho_i^{-1/2}))\leq \rho_i^{1/2}.
 	\end{equation}Since $\rho_i\rightarrow 0$, then for $i$ sufficiently large, we would get $x_i\in Y^{\reg}$ by Proposition \ref{savior}, which is a contradiction to the fact that $x_i\in Y\setminus Y^{\reg}$.
 \end{proof}
The following consequence will be crucial. Throughout the sequel, for any subset \(A\subset Y\), we denote by
\(T_r(A)\) its \(r\)-tubular neighborhood with respect to \(d_{\can}\).
 \begin{proposition}\label{thm-volume estimate for D0}
	There exists $0<\rho_0<1$ and $C>0$ such that for  all $0<r<1$, we have
	\begin{equation}\label{pathe2}
		\vol(T_r(Y\setminus Y^{\reg}),\omega_{\rm can}^m)\leq C(Y,\rho_0)r^{3-\rho_0}.
	\end{equation}
\end{proposition}
\begin{proof}
	 This estimate  follows from  Lemma \ref{lem-singular set control} and \cite[Theorem 2.4]{ABS}, which is the extension to RCD spaces of \cite[Theorem 1.3]{CN2013}.
\end{proof}

%-------------------ADDED STUFF--------------

%In addition to consider $\eta$ and $\eta'$, we should introduce an extra parameter $\eta_0$
%\begin{equation*}
%	\bar \tau\gg \eta_0\gg \eta\gg \eta'
%\end{equation*}
%Then the $\mathcal L$-geodesic would avoid an $\eta_0$-neighborhood of $D_0$ since the volume estimate \eqref{thm-volume estimate for D0} has better rate than 2. Then for this fixed $\eta_0$, we repeat our previous argument on $Y\setminus T_{\eta_0}(Y\setminus Y^{\reg})$ and we can use the H\"older bound in Theorem \ref{zatta} with $r=\eta_0$. Then as always, we allow the parameter $\eta$ and $\eta'$ sufficiently small depending on $\eta_0$.
%Then our previous argument works by taking $\eta$ and $\eta'$ sufficiently small depending on $\eta_0$.

%----------------------OLD STUFF STARTS---------------
%On the other hand, thanks to Theorem \ref{zatta}, ...... for some $\alpha>0$.

Let $\alpha>0$ be the H\"older exponent given by Theorem \ref{thm-holder continuity of potential}, and choose $\ve_0,p_0>0$ such that
\begin{equation}\label{eq--choice of epsilon0'}
		\epsilon_0=\frac{\alpha}{8(1-\alpha/2)}, \quad \frac{2}{p_0}<\epsilon_0.
			\end{equation}
We then fix $\epsilon_1$ sufficiently small such that both Theorem \ref{thm--local holder} and Lemma \ref{lem--high integrability} hold, so that we have
	\begin{equation}
		 \mathcal R_{\epsilon_1} \subset Y^{\reg},\quad \frac{\omega_{\can}^m}{\omega_Y^m}\in L^{p_0}_{\loc}(\mathcal R_{\epsilon_1}, \omega_Y^m).
	\end{equation}
Then we decompose $D\subset Y$ into a disjoint union
\begin{equation}
\begin{aligned}
	D&=\left(D\cap (Y\setminus Y^{\reg})\right)\sqcup(D\cap (Y^{\rm reg}\setminus \mathcal R_{\epsilon_1}))\sqcup (D\cap \mathcal R_{\epsilon_1})\\
&=D_0\sqcup D_1\sqcup D_2.
\end{aligned}
\end{equation}
We now define a family of open subsets $\{W_\eta\}_{\eta\in (0,\eta_0]}$ by
\begin{equation}
	W_{\eta}:=\{y\in Y\mid d_{\can}(y, D_0)<\eta\}.
\end{equation}
We fix $\eta_0$ sufficiently small so that property (4) holds. Properties (1) and (3) are obvious, and property (2') follows from Proposition \ref{thm-volume estimate for D0}.
Applying Proposition \ref{prop-with controlled events} with $W=B^{d_T}(q,\delta)$, and for $T^{-1}\ll \ov{\eta}\ll\ov{\tau}\ll \delta$ we obtain a subset $\ov{\Omega}=\ov{\Omega}(T,\ov{\eta},\bar\tau,\delta)\subset B^{d_T}(q,\delta)$ with volume
\begin{equation}\label{eq--first apply}
	\vol(\ov{\Omega},\omega(T)^n)\geq \frac{1}{2}\vol(B^{d_T}(q,\delta), \omega(T)^n)\geq c\delta^{2m}e^{-(n-m)T},
\end{equation}such that for any $q'\in \ov{\Omega}$, there exists a minimizing
 $\mathcal L$-geodesic from $(p,0)_T$ to $(q', \ov{\tau})_T$ which is disjoint from $f^{-1}(W_{\ov{\eta}})$.
 By \eqref{eq--first apply} and the estimate on the volume form \eqref{volscal}, we know that
\begin{equation}
	\vol(\ov{\Omega},\omega_0^n)\geq c>0,
\end{equation}where very importantly the constant is independent of $T$.

For this fixed value of $\ov{\eta}$, we then define a second family of open subsets $\{U_\eta\}_{\eta\in (0,\eta_0]}$ by taking $\eta_0\ll\ov{\eta}$ and defining
\begin{equation}
	U_{\eta}:=\{y\in Y\mid d_{\can}(y, \ov{D_1\setminus T_{\ov{\eta}}(D_0)})<\eta\}.
\end{equation}
We fix $\eta_0$ sufficiently small so that property (4) holds. Properties (1) and (3) are obvious, and property (2) follows from this Theorem:
%, which uses the work of Cheeger-Naber \cite{CN2013} and its extension to $\rcd$ spaces \cite{ABS}:

\begin{theorem}
	For any $\epsilon>0$ and $\rho>0$ and $0<r<1$, we have
	\begin{equation}\label{pathe}
		\vol(T_r(Y\setminus \mathcal R_{\epsilon}),\omega_{\rm can}^m)\leq C(Y,\epsilon,\rho)r^{2-\rho}.
	\end{equation}
\end{theorem}
\begin{proof}
	By the volume convergence \cite{DeG}
	and the Bishop-Gromov volume comparison, we know that there exists $\rho_0>0$ such that for any $0<r<1$ and $0<\rho<\rho_0,$
	\begin{equation}
		Y\setminus \mathcal R_{\epsilon}\subset \mathcal S_{\rho,r}^{2m-2}.
	\end{equation}%where the latter is the effective singular stratum defined in \cite[Definition 1.2]{CN2013}.
Taking $0<\rho<\rho_0$, estimate \eqref{pathe} follows from  \cite[Theorem 2.4]{ABS}.%, which is the extension to RCD spaces of \cite[Theorem 1.3]{CN2013}.
\end{proof}
The family $\{U_\eta\}$ thus satisfies the hypotheses of Proposition \ref{prop-with controlled events}, which we apply with the choices $\ve=\epsilon_0$ and $W=\ov{\Omega}$, and for $T^{-1}\ll \eta\ll\ov{\eta}\ll\ov{\tau}\ll \delta$ we obtain a subset $\Omega=\Omega(T,\eta,\ov{\eta},\bar\tau,\delta)\subset B^{d_T}(q,\delta)$ with volume
\begin{equation}
	\vol(\Omega,\omega(T)^n)\geq \frac{1}{4}\vol(B^{d_T}(q,\delta), \omega(T)^n)\geq c\delta^{2m}e^{-(n-m)T},
\end{equation}such that for any $q'\in \Omega$, there exists a minimizing
 $\mathcal L$-geodesic from $(p,0)_T$ to $(q', \ov{\tau})_T$ which has at most $\eta^{-\epsilon_0}$ $\eta$-events with respect to $\{U_\eta\}$. Again, we have
 $\vol(\Omega,\omega_0^n)\geq c>0,$ where $c$ does not depend on $T$.

For this fixed value of $\eta$, we then define a third family of open sets $\{U'_{\eta'}\}_{\eta'\in(0,\eta'_0]}$ by taking $\eta'_0\ll \eta$ and defining
\begin{equation}
	U'_{\eta'}:=\{y\in Y\mid d_{\can}(y, \overline{D\setminus (T_\eta(D_1)\cup T_{\ov{\eta}}(D_0))})<\eta'\},
\end{equation}
Observe that $\mathcal R_{\epsilon_1}$ is an open subset of $Y^{\rm reg}$ and
$\overline{D\setminus (T_\eta(D_1)\cup T_{\ov{\eta}}(D_0))}$ is a compact subset of $\mathcal R_{\epsilon_1}$, therefore we have that for $\eta'_0$ sufficiently small,
\begin{equation}
	U'_{\eta'}\subset \mathcal R_{\epsilon_1}\subset Y^{\rm reg}.
\end{equation}
Then we check that the properties (1)--(4) needed in Proposition \ref{prop-with controlled events} hold for $\{U'_{\eta'}\}_{\eta'\in (0,\eta_0']}$.
Again, we choose $\eta'_0$ sufficiently small so that property (4) holds, and properties (1) and (3) are obvious. For property (2), using the estimate in \eqref{eq--holder estimate} we see that
\begin{equation}
	U'_{\eta'}\subset V'_{C\eta'}:=\{y\in Y\mid d_{Y}(y, \overline{D\setminus (T_\eta(D_1)\cup T_{\ov{\eta}}(D_0))})<C\eta'\}.
\end{equation}
For $\eta'$ sufficiently small, $V'_{C\eta'}\subset Y^{\reg}$ is a $g_Y$-tubular neighborhood of a closed analytic subvariety inside a complex manifold (in particular, it is $(2m-2)$-rectifiable), and so the Minkowski content bound
\begin{equation}
\vol(V'_{C\eta'},\omega_Y^m)\leq C(\eta')^2,\quad 0<\eta'\leq \eta'_0,
\end{equation}
is well-known, see e.g. \cite[Theorems 3.2.39, 3.4.8]{Fe}. To prove property (2), we then use our choices of $p_0$ and $\epsilon_1$ and the H\"older inequality to bound
\begin{equation}
	\vol(U'_{\eta'},\omega_{\can}^m)\leq\left(\int_{U'_{\eta'}}\left(\frac{\omega_{\can}^m}{\omega_Y^m}\right)^{p_0}\right)^{\frac{1}{p_0}}
\vol(V'_{C\eta'},\omega_Y^m)^{1-\frac{1}{p_0}} \leq  C(Y,\eta) (\eta')^{2-\frac{2}{p_0}},
\end{equation}
Note that the constant $C$ depending on $\eta$, which is however fixed here. Then we can apply Proposition \ref{prop-with controlled events} to the
family $\{U'_{\eta'}\}$ and $W=\Omega$, where $\Omega$ here denotes the set obtained in the last step, and we obtain:

\begin{proposition}\label{judicious}
For $T^{-1}\ll \eta'\ll \eta\ll\ov{\eta}\ll\ov{\tau}\ll \delta$ there exists a subset $\Omega'\subset\Omega\subset\ov{\Omega} \subset B^{d_T}(q,\delta)$ with
\begin{equation}\label{eq-volume lower bound of good geodesics}
	\vol(\Omega',\omega(T)^n)\geq \frac{1}{8}\vol(B^{d_T}(q,\delta), \omega(T)^n),
 \end{equation}
such that for any $q'\in\Omega'$  there exists a minimizing $\mathcal L$-geodesic from $(p,0)_T$ to $(q', \ov{\tau})_T$ which has at most $(\eta')^{-\epsilon_0}$ $\eta'$-events with respect to $\{U'_{\eta'}\}$, and also at most $\eta^{-\epsilon_0}$ $\eta$-events with respect to  $\{U_{\eta}\}$, and is disjoint from $f^{-1}(W_{\ov{\eta}})$.
\end{proposition}

\subsection{Completion of the proof}
We are now ready to complete the proof of Theorem \ref{main}, by proving the key estimate \eqref{key}. We adapt an idea from \cite[Remark 4.2]{LT}, but using crucially our results from the previous section.

For this, we apply Proposition \ref{judicious}, and let $\gamma$ be any of the minimizing $\mathcal{L}$-geodesics given by that, with endpoints in $\Omega'$, which is disjoint from $f^{-1}(W_{\ov{\eta}})$. We consider subsets $I\cup I'$ of  $[0, \ov{\tau}]$ defined by the property that $\tau \in I \Leftrightarrow \gamma(\tau) \in \tilde{U}_\eta=f^{-1}(U_{\eta})$ and $\tau \in I' \Leftrightarrow \gamma(\tau) \in \tilde{U}'_{\eta'}=f^{-1}(U'_{\eta'})$, and its complement
 \begin{equation}J=[0, \ov{\tau}] \backslash (I\cup I').\end{equation}By Lemma \ref{lem--lower bound of first bad time}, we know that every $\tau \in I\cup I'$ satisfies $\tau \geqslant C^{-1} \ov{\tau}$.

In the following we estimate $\int_I\left|\partial_\tau \gamma\right|_{\tilde{g}(\tau)} d \tau$ and $\int_{I'}\left|\partial_\tau \gamma\right|_{\tilde{g}(\tau)} d \tau.$
First, as in the proof of Proposition \ref{prop-with controlled events}, we call $\Omega'_{p}\subset T_pX$ the open subset which under $\mathcal{L}{\rm exp}_{p,\ov{\tau}}$ maps diffeomorphically onto an open subset of $\Omega'$ with full measure. As in \eqref{eq--fubini} we obtain
	\begin{equation}\label{eq--fubini2}
		\begin{aligned}
		\int_{C^{-1}\ov{\tau}}^{\ov{\tau}}\vol(\tilde U_\eta,\tilde \omega(\tau)^n)d\tau &\geq C^{-1}e^{-\frac{C}{\ov{\tau}}}|I| \vol(\Omega',\omega(T)^n),
		\end{aligned}	\end{equation}
and using the volume bounds in \eqref{ut} and \eqref{eq-volume lower bound of good geodesics}, and repeating this argument for $I'$, we see that there exists at least one of these $\mathcal L$-geodesics $\gamma$ (obtained through Proposition \ref{judicious}) such that
\begin{equation}\label{eq--upper bound on length}
\begin{aligned}
		|I|& \leqslant C \ov{\tau} \eta^{2-\epsilon_0} \delta^{-2 m} e^{C / \ov{\tau}},\\
		|I'|&\leq C \ov{\tau} \eta^{\prime 2-\epsilon_0} \delta^{-2 m} e^{C / \ov{\tau}},
\end{aligned}
\end{equation}
and is disjoint from $f^{-1}(W_{\ov{\eta}})$.
  Then using the scalar curvature bound \eqref{stscal} and the upper bound on the $\mathcal L$-distance \eqref{eq--upper bound of ldistance}, we can estimate
\begin{equation}\label{eq--length on I}
	\begin{aligned}
\int_I\left|\partial_\tau \gamma\right|_{\tilde{g}(\tau)} d \tau
& \leqslant C \ov{\tau}^{-\frac{1}{4}}\left(\int_I \sqrt{\tau}\left|\partial_\tau \gamma\right|_{\tilde{g}(\tau)}^2 d \tau\right)^{\frac{1}{2}}|I|^{\frac{1}{2}} \\
& \leqslant C(\delta,\bar\tau)\eta^{1-\epsilon_0}\left( \ov{\tau}^{\frac{3}{2}}+\int_I \sqrt{\tau}\left(R(\tilde{g}(\tau))+\left|\partial_\tau \gamma\right|_{\tilde{g}(\tau)}^2\right) d \tau\right)^{\frac{1}{2}} \\
& \leqslant  C(\delta,\bar\tau)\eta^{1-\epsilon_0}\left( \ov{\tau}^{\frac{3}{2}}+ \ov{\tau}^{-\frac{1}{2}}\right)^{\frac{1}{2}}=\Psi(\eta|\delta,\bar\tau) \eta^{\frac{1}{2}},
\end{aligned}
\end{equation}Similarly by taking $\eta'$ even smaller, we have
\begin{equation}\label{length on I'}
	\int_{I'}\left|\partial_\tau \gamma\right|_{\tilde{g}(\tau)} d \tau \leq \Psi(\eta'|\delta,\bar\tau,\eta)(\eta')^{\frac12}.
\end{equation}

  By construction, there are at most $\eta^{-\epsilon_0}$ $\eta$-events with respect to $\{U_\eta\}$, which we label by
$[\tau_{\entry,i}, \tau_{\exit,i}], i\in \mathcal{I}, |\mathcal{I}|\leq \eta^{-\ve_0}$ and let $I_{\eta}=\bigcup_{i\in\mathcal{I}}[\tau_{\entry,i}, \tau_{\exit,i}]\subset I$. Then for every $i\in\mathcal{I}$ we have
\begin{equation}
f\bigl(\gamma([\tau_{\entry,i}, \tau_{\exit,i}])\bigr) \subset U_{\eta},\quad f\bigl(\gamma([0,\bar\tau]\setminus I_{\eta} )\bigr) \subset Y\setminus U_{\eta/2}.
\end{equation}
Similarly, there exist at most $\eta'^{-\epsilon_0}$ intervals, which we label by $[\tau'_{\entry,i}, \tau'_{\exit,i}], i\in\mathcal{I}',|\mathcal{I}'|\leq \eta'^{-\ve_0},$
such that
\begin{equation}
f\bigl(\gamma([\tau'_{\entry,i}, \tau'_{\exit,i}])\bigr) \subset U'_{\eta'},\quad f\bigl(\gamma([0,\bar\tau]\setminus \bigcup_{i\in\mathcal{I}'}[\tau'_{\entry,i}, \tau'_{\exit,i}] )\bigr) \subset Y\setminus U'_{\eta'/2}.
\end{equation}

From the homeomorphism property in Theorem \ref{thm--rcd}, it follows that $d_{\can}$ is continuous with respect to $d_Y$, and so we see that there exists $r=r(\bar\eta)$ such that
\begin{equation}
	\{y\in Y\mid d_Y(y, Y\setminus Y^{\reg})\leq 2r\}\subset W_{\bar\eta}
\end{equation}
and since $f(\gamma)$ is disjoint from $W_{\ov{\eta}}$, we can apply Theorem \eqref{zatta} with $r(\ov{\eta})$ and obtain the H\"older estimate \eqref{holder} (with constant $C=C(\ov{\eta},\ov{\tau},\delta)$ and uniform constant $\alpha$). This implies that for every $i\in\mathcal{I}$ we can find a curve in $Y$ joining $f(\gamma(\tau_{\entry,i}))$ and $f(\gamma(\tau_{\exit,i}))$ whose $g_{\rm can}$-length is at most
\begin{equation}\label{path replacement}
\begin{aligned}
C (d_Y(f(\gamma(\tau_{\entry,i})),f(\gamma(\tau_{\exit,i})))^{\frac{\alpha}{2}}&\leq C\left(\int_{\tau_{\entry,i}}^{\tau_{\exit,i}}\left|\partial_\tau f(\gamma)\right|_{g_Y}d\tau \right)^{\frac{\alpha}{2}}\\
&\leq C \left(\int_{\tau_{\text {entry},i}}^{\tau_{\text {exit},i}}\left|\partial_\tau \gamma\right|_{\tilde{g}(\tau)} d \tau\right)^{\frac{\alpha}{2}},
\end{aligned}
\end{equation}
where we used \eqref{schwarz}.

We then perform a replacement procedure as follows. For each interval $[\tau'_{\entry,i}, \tau'_{\exit,i}], i\in\mathcal{I}'$, we consider a sub-interval (which could be empty) as follows. Let
\begin{equation}
\begin{aligned}
	\tau''_{\entry,i}&=\inf\{t\in [\tau'_{\entry,i}, \tau'_{\exit,i}]\mid t\notin I_{\eta}\},\\
	\tau''_{\exit,i}&=\sup\{t\in [\tau'_{\entry,i}, \tau'_{\exit,i}]\mid t\notin I_{\eta}\},
\end{aligned}
\end{equation}and
\begin{equation}
	 I_{\eta'}=\bigcup_{i\in\mathcal{I}'}[\tau''_{\entry,i},\tau''_{\exit,i}]\subset I'.
\end{equation}
We then delete the subintervals of $I_{\eta}$ that are contained in $I_{\eta'}$, i.e. we consider the subset $\ti{\mathcal I} \subset \mathcal{I}$ such that $i\in \ti{\mathcal I} $ if and only if $[\tau_{\entry,i},\tau_{\exit,i}]\subset  I_{\eta'},$
and define
\begin{equation}
	\widetilde I_{\eta}=\bigcup_{i\in\mathcal{I}\backslash\ti{ \mathcal I}}[\tau_{\entry,i},\tau_{\exit,i}]\subset I_{\eta}.
\end{equation}After this procedure, we obtain that  the interiors of $I_{\eta'}$ and $\widetilde I_{\eta}$ are disjoint, and we still have that $|\mathcal{I}'|\leq\eta'^{-\ve_0}$, and $|\mathcal{I}\backslash\ti{ \mathcal I}|\leq\eta^{-\ve_0}$.
Then we use the paths obtained in \eqref{path replacement} to replace the portion of $f(\gamma)$ with $\tau_{\entry,i}\leq \tau\leq \tau_{\exit,i}$ for $i\in \mathcal{I}\backslash\ti{ \mathcal I}$ and similarly for the $\tau''_i$'s with $i\in \mathcal{I}'$. We obtain a curve $\widetilde\gamma$ in $Y$ from $f(p)$ to $f(q)$ satisfying
\begin{equation}
	\begin{aligned}
&d_{\can}(f(p),f(q))\leq \mathrm{length}_{g_{\text {can }}}\left(\widetilde\gamma\right)\\
 & \leqslant \int_J\left|\partial_\tau f(\gamma)\right|_{g_{\text {can }}} d \tau+C \sum_{i=1}^{\eta^{-\epsilon_0}}\left(\int_{\tau_{\text {entry}, i}}^{\tau_{\text {exit}, i}}\left|\partial_\tau \gamma\right|_{\tilde{g}(\tau)} d \tau\right)^{\frac{\alpha}{2}} \\
&+C \sum_{i=1}^{(\eta')^{-\epsilon_0}}\left(\int_{\tau''_{\text {entry}, i}}^{\tau''_{\text{exit}, i}}\left|\partial_\tau \gamma\right|_{\tilde{g}(\tau)} d \tau\right)^{\frac{\alpha}{2}}+\Psi(T^{-1},\delta).
\end{aligned}
\end{equation}
Using \eqref{twyconv}, and the H\"older inequality for the sums, we get \begin{equation}
	\begin{aligned}
		d_{\can}(f(p),f(q))
\leqslant &  (1+\Psi(T^{-1}|\bar\tau))\int_J\left|\partial_\tau \gamma\right|_{\tilde{g}(\tau)} d \tau+ \Psi(T^{-1},\delta)\\
&+C \eta^{-\epsilon_0(1-\frac{\alpha}{2})}\left(\int_I\left|\partial_\tau \gamma\right|_{\tilde{g}(\tau)} d \tau\right)^{\frac{\alpha}{2}} \\
&+C (\eta')^{-\epsilon_0(1-\frac{\alpha}{2})}\left(\int_{I'}\left|\partial_\tau \gamma\right|_{\tilde{g}(\tau)} d \tau\right)^{\frac{\alpha}{2}} \\
	\end{aligned}
\end{equation}
Then using \eqref{eq--length on I} and \eqref{length on I'}, we obtain
\begin{equation}
	\begin{aligned}
		d_{\can}(f(p),f(q))
 \leqslant &(1+\Psi(T^{-1}|\bar\tau))\int_0^{\ov{\tau}}\left|\partial_\tau \gamma\right|_{\tilde{g}(\tau)} d \tau+\Psi(T^{-1},\delta)\\
&+\Psi(\eta|\delta,\ov{\eta},\bar\tau)  \eta^{\frac{\alpha}{4}-\epsilon_0(1-\frac{\alpha}{2})}+\Psi(\eta'|\delta,\ov{\eta},\bar\tau,\eta)(\eta')^{\frac{\alpha}{4}-\epsilon_0(1-\frac{\alpha}{2})} \\
 \leqslant & (1+\Psi(T^{-1}|\bar\tau))\int_0^{\ov{\tau}}\left|\partial_\tau \gamma\right|_{\tilde{g}(\tau)} d \tau+ \Psi(T^{-1},\delta)\\
 &+\Psi(\eta,T^{-1}|\delta,\ov{\eta},\bar\tau)+\Psi(\eta',T^{-1}|\eta,\ov{\eta},\delta,\bar\tau),
	\end{aligned}
\end{equation}
 where in the last inequality we have crucially used our choice of $\epsilon_0$ in \eqref{eq--choice of epsilon0'}.

Then we use the fact that $\gamma$ is a minimizing $\mathcal L$-geodesic from $(p,0)_T$ to $(q',\ov{\tau})_T$ to estimate
\begin{equation}
\begin{aligned}
\int_0^{\ov{\tau}}\left|\partial_\tau \gamma\right|_{\tilde{g}(\tau)} d \tau & \leqslant\left(\int_0^{\ov{\tau}} \sqrt{\tau}\left|\partial_\tau \gamma\right|_{\tilde{g}(\tau)}^2 d \tau\right)^{\frac{1}{2}}\left(\int_0^{\ov{\tau}} \frac{1}{\sqrt{\tau}} d \tau\right)^{\frac{1}{2}} \\
& \leqslant \sqrt{2} \ov{\tau}^{\frac{1}{4}}\left(C \ov{\tau}^{\frac{3}{2}}+\int_0^{\ov{\tau}} \sqrt{\tau}\left(R(\tilde{g}(\tau))+\left|\partial_\tau \gamma\right|_{\tilde{g}(\tau)}^2\right) d \tau\right)^{\frac{1}{2}} \\
& =\sqrt{2} \ov{\tau}^{\frac{1}{4}}\left(C \ov{\tau}^{\frac{3}{2}}+L_T\left(q^{\prime}, \ov{\tau}\right)\right)^{\frac{1}{2}}\\
&=\sqrt{2} \ov{\tau}^{\frac{1}{4}}\left(C \ov{\tau}^{\frac{3}{2}}+L_T\left(q, \ov{\tau}\right)+\Psi(\delta)\bar\tau^{-1/2}\right)^{\frac{1}{2}}
\end{aligned}	
\end{equation}
Therefore, by taking $\delta$ small, and then taking $\eta'\ll\eta\ll\ov{\eta}\ll\ov{\tau}\ll\delta$ and then $T$ sufficiently large, we have finally shown that
\begin{equation}
L_T(q,\ov{\tau})\geq \frac{1}{2\sqrt{\ov{\tau}}}d_{\can}(f(p),f(q))^2+\Psi(T^{-1}|\ov{\tau})+\Psi(\bar\tau),
\end{equation}which is precisely \eqref{key}.


\begin{thebibliography}{99}
\bibitem{ABS} G. Antonelli, E. Bru\`e, D. Semola, {\em Volume bounds for the quantitative singular strata of non-collapsed RCD metric measure spaces}, Anal. Geom. Metr. Spaces {\bf 7} (2019), no. 1, 158--178.
%\bibitem{BG} E. Bombieri, E. Giusti, {\em Harnack's inequality for elliptic differential equations on minimal surfaces},  Invent. Math. {\bf 15} (1972), 24--46.
\bibitem{Cao} H.-D. Cao, {\em Deformation of K\"ahler metrics to K\"ahler-Einstein metrics on compact K\"ahler manifolds}, Invent. Math. {\bf 81}  (1985), no. 2, 359--372.
\bibitem{CDS2} X. Chen, S.K. Donaldson, S. Sun, {\em K\"ahler-Einstein metrics on Fano manifolds. II: Limits with cone angle less than $2\pi$}, J. Amer. Math. Soc. {\bf 28} (2015), no. 1, 199--234.
\bibitem{CJN} J. Cheeger, W. Jiang, A. Naber, {\em Rectifiability of singular sets of noncollapsed limit spaces with Ricci curvature bounded below}, Ann. of Math. (2) {\bf 193} (2021), no. 2, 407--538.
\bibitem{CN2013} J. Cheeger, A. Naber, {\em Lower bounds on Ricci curvature and quantitative behavior of singular sets}, Invent. Math. {\bf 191} (2013), no. 2, 321--339.
\bibitem{CN2012} T.H. Colding, A. Naber, {\em Sharp H\"older continuity of tangent cones for spaces with a lower Ricci curvature bound and applications},  Ann. of Math. (2) {\bf 176} (2012), 1173--1229.
\bibitem{CGZ} D. Coman, V. Guedj, A. Zeriahi, {\em Extension of plurisubharmonic functions with growth control}, J. Reine Angew. Math. {\bf 676} (2013), 33--49.
\bibitem{deng2015} Q. Deng, {\em H\"older continuity of tangent cones in RCD$(K,N)$ spaces and applications to nonbranching}, Geom. Topol. {\bf 29} (2025), no. 2, 1037--1114.
\bibitem{DeG} G. De Philippis, N. Gigli, {\em Non-collapsed spaces with Ricci curvature bounded from below}, J. \'Ec. polytech. Math. {\bf 5} (2018), 613--650.
\bibitem{DZ} S. Dinew, Z. Zhang, {\em On stability and continuity of bounded solutions of degenerate complex Monge-Amp\`ere equations over compact K\"ahler manifolds}, Adv. Math. {\bf 225} (2010), no. 1, 367--388.
\bibitem{DS1} S.K. Donaldson, S. Sun, {\em Gromov-Hausdorff limits of K\"ahler manifolds and algebraic geometry}, Acta Math. {\bf 213} (2014), no. 1, 63--106.
\bibitem{DS2} S.K. Donaldson, S. Sun, {\em Gromov-Hausdorff limits of K\"ahler manifolds and algebraic geometry, II}, J. Differential Geom. {\bf 107} (2017), no. 2, 327--371.
\bibitem{EGZ} P. Eyssidieux, V. Guedj, A. Zeriahi, {\em Singular K\"ahler-Einstein metrics}, J. Amer. Math. Soc. {\bf 22} (2009), 607--639.
\bibitem{Fe} H. Federer, {\em Geometric measure theory}, Die Grundlehren der mathematischen Wissenschaften, Band 153. Springer-Verlag New York, Inc., New York, 1969.
\bibitem{Fi} S. Finski, {\em On the metric structure of section ring}, preprint, arXiv:2209.03853.
\bibitem{GTZ} M. Gross, V. Tosatti, Y. Zhang, \emph{Collapsing of abelian fibered Calabi-Yau manifolds}, Duke Math. J. {\bf 162} (2013), no. 3, 517--551.
\bibitem{GTZ2} M. Gross, V. Tosatti, Y. Zhang, {\em Geometry of twisted K\"ahler-Einstein metrics and collapsing},  Comm. Math. Phys. {\bf 380} (2020), no.3, 1401--1438.
\bibitem{GGZ} V. Guedj, H. Guenancia, A. Zeriahi, {\em Diameter of K\"ahler currents}, J. Reine Angew. Math. {\bf 820} (2025), 115--152.
\bibitem{guedjT} V. Guedj, T.D. T{\^o}, {\em K\"ahler families of Green's functions}, J. \'Ec. polytech. Math. {\bf 12} (2025), 319--339.
\bibitem{Gu} B. Guo, {\em On the K\"ahler Ricci flow on projective manifolds of general type},  Int. Math. Res. Not. IMRN 2017, no. 7, 2139--2171.
\bibitem{GKSS} B. Guo, S. Ko\l odziej, J. Song, J. Sturm, {\em H\"older estimates for degenerate complex Monge-Amp\`ere equations}, preprint, arXiv:2508.20933.
\bibitem{GPSS} B. Guo, D.H. Phong, J. Song, J. Sturm, {\em Diameter estimates in K\"ahler geometry}, Comm. Pure Appl. Math. {\bf 77} (2024), no. 8, 3520--3556.
\bibitem{GPSS2} B. Guo, D.H. Phong, J. Song, J. Sturm, {\em Diameter estimates in K\"ahler geometry II: removing the small degeneracy assumption}, Math. Z. {\bf 308} (2024), no. 3, Paper No. 43, 7 pp.
\bibitem{GPSS3} B. Guo, D.H. Phong, J. Song, J. Sturm, {\em Sobolev inequalities on K\"ahler spaces}, preprint, arXiv:2311.00221.
\bibitem{GSW} B. Guo, J. Song, B. Weinkove, {\em Geometric convergence of the K\"ahler-Ricci flow on complex surfaces of general type},  Int. Math. Res. Not. IMRN 2016, no. 18, 5652--5669.
\bibitem{HLT} H.-J. Hein, M.-C. Lee, V. Tosatti, {\em Collapsing immortal K\"ahler-Ricci flows},  Forum Math. Pi {\bf 13} (2025), Paper No. e18.
\bibitem{Kol} S. Ko\l odziej, {\em H\"older continuity of solutions to the complex Monge-Amp\`ere equation with the right-hand side in $L^p$: the case of compact K\"ahler manifolds}, Math. Ann. {\bf 342} (2008), no. 2, 379--386.
\bibitem{JS} W. Jian, J. Song, {\em Diameter estimates for long-time solutions of the K\"ahler-Ricci flow}, Geom. Funct. Anal. {\bf 32} (2022), no. 6, 1335--1356.
\bibitem{LiT} C. Li, {\em K\"ahler-Einstein metrics and K-stability}, PhD thesis, Princeton University, 2012.
\bibitem{Li} Y. Li, {\em On collapsing Calabi-Yau fibrations}, J. Differential Geom. {\bf 117} (2021), no. 3, 451--483.
\bibitem{LT} Y. Li, V. Tosatti, {\em On the collapsing of Calabi-Yau manifolds and K\"ahler-Ricci flows}, J. Reine Angew. Math. {\bf 800} (2023), 155--192.
\bibitem{LiuSz} G. Liu, G. Sz\'ekelyhidi, {\em  Gromov-Hausdorff limits of K\"ahler manifolds with Ricci curvature bounded below}, Geom. Funct. Anal. {\bf 32} (2022), 236--279.
\bibitem{perelman} G. Perelman, {\em The entropy formula for the Ricci flow and its geometric applications}, preprint, arXiv:math/0211159.
\bibitem{So} J. Song, {\em Riemannian geometry of K\"ahler-Einstein currents}, preprint, arXiv:1404.0445.
\bibitem{ST0} J. Song, G. Tian, {\em The K\"ahler-Ricci flow on surfaces of positive Kodaira dimension}, Invent. Math. {\bf 170} (2007), no. 3, 609--653.
\bibitem{ST} J. Song, G. Tian, {\em Canonical measures and K\"ahler-Ricci flow}, J. Amer. Math. Soc. {\bf 25} (2012), no. 2, 303--353.
\bibitem{ST2} J. Song, G. Tian, {\em Bounding scalar curvature for global solutions of the K\"ahler-Ricci flow}, Amer. J. Math. {\bf 138} (2016), no. 3, 683--695.
\bibitem{ST3} J. Song, G. Tian, {\em The K\"ahler-Ricci flow through singularities}, Invent. Math. {\bf 207} (2017), no. 2, 519--595.
\bibitem{STZ} J. Song, G. Tian, Z. Zhang, {\em Collapsing behavior of Ricci-flat K\"ahler metrics and long time solutions of the K\"ahler-Ricci flow}, preprint, arXiv:1904.08345.
\bibitem{stol} G. Stolzenberg, {\em  Volumes, limits, and extensions of analytic varieties}, Lecture Notes in Mathematics, No. 19. Springer-Verlag, Berlin-New York, 1966.
\bibitem{Sz2} G. Sz\'ekelyhidi, {\em Singular K\"ahler-Einstein metrics and RCD spaces}, Forum Math. Pi {\bf 13} (2025), Paper No. e24, 33 pp.
\bibitem{Sz} G. Sz\'ekelyhidi, {\em Gromov-Hausdorff limits of collapsing Calabi-Yau fibrations}, preprint, arXiv:2505.14939.
\bibitem{TiZh} G. Tian, Z. Zhang, {\em On the K\"ahler-Ricci flow on projective manifolds of general type}, Chinese Ann. Math. Ser. B {\bf 27} (2006), no. 2, 179--192.
\bibitem{TiZ} G. Tian, Z. Zhang, {\em Convergence of K\"ahler-Ricci flow on lower dimensional algebraic manifolds of general type},  Int. Math. Res. Not. IMRN 2016, no. 21, 6493--6511.
\bibitem{Ts} H. Tsuji, \emph{Existence and degeneration of K\"ahler-Einstein metrics on minimal algebraic varieties of general type}, Math. Ann. {\bf 281} (1988), no. 1, 123--133.
\bibitem{To} V. Tosatti, {\em Immortal solutions of the K\"ahler-Ricci flow}, to appear in Contemp. Math.
\bibitem{TWY2} V. Tosatti, B. Weinkove, X. Yang, {\em Collapsing of the Chern-Ricci flow on elliptic surfaces},  Math. Ann. {\bf 362} (2015), no. 3-4, 1223--1271.
\bibitem{TWY} V. Tosatti, B. Weinkove, X. Yang, {\em The K\"ahler-Ricci flow, Ricci-flat metrics and collapsing limits},  Amer. J. Math. {\bf 140} (2018), no. 3, 653--698.
\bibitem{TZ3} V. Tosatti, Y. Zhang, {\em Finite time collapsing of the K\"ahler-Ricci flow on threefolds}, Ann. Sc. Norm. Super. Pisa Cl. Sci. {\bf 18} (2018), no.1, 105--118.
\bibitem{varolin} D. Varolin, {\em Division theorems and twisted complexes}, Math. Z. {\bf 259} (2008), no. 1, 1--20.
\bibitem{vu2024} D.V. Vu, {\em Uniform diameter and non-collapsing estimates for K\"ahler metrics}, J. Geom. Anal. {\bf 36} (2026), no. 2, Paper No. 75.
\bibitem{Wang} B. Wang, {\em The local entropy along Ricci flow. Part A: the no-local-collapsing theorems}, Camb. J. Math. {\bf 6} (2018), no.3, 267--346.
\bibitem{ZhK} K. Zhang, {\em Some refinements of the partial $C^0$ estimate}, Anal. PDE {\bf 14} (2021), no. 7, 2307--2326.
\bibitem{Zh} Z. Zhang, {\em Scalar curvature bound for K\"ahler-Ricci flows over minimal manifolds of general type}, Int. Math. Res. Not. IMRN 2009, no. 20, 3901--3912.
\end{thebibliography}
\end{document}